\documentclass[11pt]{article}
\usepackage{amsmath,amssymb,amsfonts}
\usepackage{enumerate}
\usepackage[latin1]{inputenc}
\allowdisplaybreaks

\author{
Fazia {\sc Bedouhene}\footnote{Mouloud Mammeri University of Tizi-Ouzou,
Laboratoire de Math\'ematiques Pures et Appliqu\'ees, Tizi-Ouzou, Algeria
E-Mail: fbedouhene@yahoo.fr
ORCID ID 0000-0002-2664-2445
},
Nouredine {\sc Challali}\footnote{
Mouloud Mammeri University of Tizi-Ouzou,
Laboratoire de Math\'ematiques Pures et Appliqu\'ees, Tizi-Ouzou, Algeria
E-Mail: challalin@yahoo.fr
ORCID ID 0000-0002-6965-5086
},
Omar {\sc Mellah}
\footnote{Mouloud Mammeri University of Tizi-Ouzou,
Laboratoire de Math\'ematiques Pures et Appliqu\'ees, Tizi-Ouzou, Algeria
E-Mail: omellah@yahoo.fr
ORCID ID 0000-0002-3042-9719
},\\
Paul {\sc Raynaud de Fitte}
\footnote{Normandie Univ, Laboratoire Rapha\"el Salem,
UMR CNRS 6085, Rouen, France
E-Mail: prf@univ-rouen.fr
ORCID ID 0000-0001-5527-9393},
and
Mannal {\sc Smaali}\footnote{
Mouloud Mammeri University of Tizi-Ouzou,
Laboratoire de Math\'ematiques Pures et Appliqu\'ees, Tizi-Ouzou, Algeria
E-Mail: smaali\_manel@yahoo.fr
ORCID ID 0000-0002-4556-8930},
}

\title{Almost periodic solution in distribution for stochastic differential equations with Stepanov
 almost periodic coefficients}

\newtheorem{theo}{Theorem}[section]
\newtheorem{prop}[theo]{Proposition}
\newtheorem{cor}[theo]{Corollary}
\newtheorem{lemma}[theo]{Lemma}
\newtheorem{definition}[theo]{Definition}

\newtheorem{remark}[theo]{Remark}

\newtheorem{example}[theo]{Example}

\def\cprime{$'$}

\renewcommand\epsilon{\varepsilon}

\newcommand\N{\mathbb{N}}
\newcommand\R{\mathbb{R}}

\newcommand\cS{\mathcal{S}}

\newcommand\un[1]{\,\rlap{{\rm 1}}\kern.22em \mbox{\rm l}_{#1}} 

\newcommand\meas{\mathop{\text{\rm meas}}  } 

\newcommand\proof{\noindent {\bf Proof}\ }
\newcommand\proofof[1]{\noindent{\bf Proof of {#1}}\ }
\renewcommand\square{\fbox{\rule{0em}{.3em}\rule{.3em}{0em}} \qquad}
\newcommand\finpr{\hfill$\square\qquad$\medskip\par}

\newcommand\tq{;\,} 

\newcommand\CCO[1]{\left( #1 \right)}
\newcommand\norm[1]{\left\Vert #1 \right\Vert}
\newcommand\abs[1]{\left\vert #1 \right\vert}
\newcommand\accol[1]{\left\{#1\right\}}

\newcommand\expect{\mathop{\text{\rm E}}\nolimits}

\newcommand\loi{\mathop{\text{\rm law}}}
\newcommand\law[1]{ \loi({#1}) } 

\newcommand\lr{\mbox{\tiny\rm L}}

\newcommand\ellp[1]{\mathop{\text{\rm L}}\nolimits^{#1}}

\newcommand\elp[1]{{\mathfrak{L}}^{#1}} 

\newcommand\trace{\mathop{\mbox{\rm tr}}}
\newcommand\Dom{\mathop{\mbox{\rm Dom}}}

\newcommand\esprob{\Omega}
\newcommand\tribu{\mathcal{F}}
\newcommand\prob{\mathop{\text{\rm P}}\nolimits}
\newcommand\laws[1]{{\mathcal M}^{1,+}\CCO{{#1}}}

\newcommand\transl[1]{\widetilde{#1}}

\newcommand\C{\mathcal{C}}
\newcommand\h{\mathbb{H}}

\newcommand\CUB{\mbox{\rm CUB}}
\newcommand\espmes{\mathbb{U}}
\newcommand\trmes{\Sigma}

\newcommand\elmes{u}
\newcommand\fgg{\tilde{F}}
\newcommand\ff{F'}
\newcommand\fg{F}
\newcommand\mdense{\mathcal{D}}

\newcommand\mnegl{\mathcal{N}}

\newcommand\espB{\mathbb{B}}

\newcounter{deplus}
\newcommand\fK{K'}
\newcommand\fKg{K}
\newcommand\elmess{v}
\newcommand\Mmes{A}
\newcommand\fgC{G}
\newcommand\fKgC{L}

\newcommand\D{\mathcal{D}} 

\newcommand\bl{\mbox{\tiny\rm BL}}
\newcommand\WASS{\mathop{\mbox{\textrm{Wass}}}\nolimits}

\newcommand\ER{\mathcal{E}}     
\newcommand\AP{\mbox{\rm{AP}}}
\newcommand\StER[2][]{\St^{#2}_{#1}\mathcal{E}}
\newcommand\PAP{\mbox{\textrm{PAP}}}
\newcommand\StPAP[2][]{{\St^{#2}_{#1}}\mbox{\textrm{PAP}}}
\newcommand\StAP[2][]{\St^{#2}_{#1}\!\mbox{\rm{AP}}} 

\newcommand\StAPUb[2][]{\St^{#2}\!\rm{APU}_b} 
\newcommand\StAPUc[2][]{\St^{#2}\!\rm{APU}_c} 
\newcommand\StPAPUc[2][]{\St^{#2}\rm{PAPU}_c} 

\newcommand\Cont{\text{\rm C}} 
\newcommand\BCont{\text{\rm BC}} 

\newcommand\StERUc[2][]{\St^{#2}\mathcal{E}\mbox{\textrm{U}}_c}  

\newcommand\PAPUc{\mbox{\textrm{PAPU}}_c}
\newcommand\APUc{\mbox{\textrm{APU}}_c}
\newcommand\APUb{\mbox{\textrm{APU}}_b} 

\newcommand\APD{\mbox{\textrm{APD}}} 
\newcommand\APOD{\mbox{\textrm{APD}}_1}

\newcommand\PAPD{\mbox{\textrm{PAPD}}}

\newcommand\St{\mathbb{S}} 

\newcommand\espE{{\mathbb E}}
\newcommand\espX{{\mathbb X}}
\newcommand\espY{{\mathbb Y}}
\newcommand\espL{\textrm{L}} 
\newcommand\espLH{\espL_2} 
\newcommand\lip{\rm{(Lip)}}

\newcommand{\verti}[1]{{\left\vert\kern-0.25ex\left\vert\kern-0.25ex\left\vert #1
    \right\vert\kern-0.25ex\right\vert\kern-0.25ex\right\vert}}

\newcommand\Trunc[2][1]{{\lfloor {#2},x_{0}\rfloor}} 

\newcommand\ii{\tau}

\begin{document}

\maketitle

\begin{abstract}
This paper deals with the existence and uniqueness of ($\mu$-pseudo) almost periodic mild solution to some evolution equations with Stepanov ($\mu$-pseudo) almost periodic coefficients, in both determinist and stochastic cases. After revisiting some known concepts and properties of Stepanov ($\mu$-pseudo) almost periodicity in  complete metric space, we consider a semilinear stochastic evolution equation  on a Hilbert separable space with  Stepanov ($\mu$-pseudo) almost periodic coefficients. We show existence and uniqueness of the mild solution which is ($\mu$-pseudo) almost periodic in $2$-distribution. We also generalize a result by Andres and Pennequin, according to which there is no purely Stepanov almost periodic solutions to differential equations with Stepanov almost periodic coefficients.
\end{abstract}



\section{Introduction}

The concept of Stepanov almost periodicity, which is the central issue in this paper, was first introduced in the literature by Stepanov \cite{Stepanoff1926}, and is a natural generalization of the concept of almost periodicity in Bohr's sense. Important contributions upon such concept where subsequently made by N. Wiener \cite{Wiener1926}, P. Franklin \cite{Franklin1929}, A. S. Besicovitch \cite{Besicovitch-1955-book}, B. M. Levitan and V. V. Zhikov \cite{Levitan53,Levitan_Zhikov82}, Amerio and Prouse \cite{Amerio-Prouse-1971-book}, S. Zaidman \cite{Zaidmann71_Existence_Stepanov}, A. S. Rao \cite{rao_almost_1999}, C. Corduneanu \cite{Corduneanu89book}, L. I. Danilov \cite{Danilov-1997,Danilov98-uniform-approximation}, S. Stoi{\'n}ski \cite{stoinski_remarks_1996,Stoinski_AP_Lebesgue_measure94}, J. Andres, A. M. Bersani, G. Grande, K. Lesniak \cite{Andres99,andres_almost-periodicity_2001,Andres-Bersani-Grande-06-Hierarchy,Andres_Bersani_Lesniak_01_APvarious_metric}.  Outside the field of harmonic analysis, a substantial application of Stepanov almost periodicity lies in the theory of differential equations \cite{Nemytskii_Stepanov_book_60}. In this context, the theory of dynamical systems is pertinent and, in particular, in the study of various kinds and extensions of almost periodic and (or) almost automorphic motions. This is due mainly to their importance and applications in physical sciences. One can mention e.g. T. Diagana \cite{diagana06weighted,diagana08weighted,diagana09stepanov,diagana_existence_2009_SAA}, J. Blot et al. \cite{blot-al09superposition,blot-cieutat_ezzinbi2012,blot-cieutat-ezzinbi2013,blot-mophou-nguerekata-pennequin2009}, G. M. N'Gu{\'e}r{\'e}kata et al. \cite{nguerekata05book,nguerekata-pankov08stepanov}, J. Andres and D. Pennequin \cite{Andres-Pennequin2012,andres_stepanov_2012}, Z. Hu and A. B. Mingarelli \cite{hu_boundeness_2005,Hu-Mingarelli08,hu_stepanov-like_2012}.

Though there has been a significant attention devoted to the theory of Stepanov almost periodicity in the deterministic case, there are few works related to the notion of Stepanov almost periodicity for stochastic processes. To our knowledge, the first work dedicated to Stepanov almost periodically correlated (APC) processes is due to L. H. Hurd and A. Russek \cite{hurd_stepanov_1996}, where Gladyshev's characterization of
APC correlation functions was extended to Stepanov APC processes. In the framework of Stochastic differential equation, Bezandry and Diagana \cite{Bezandry-Diagana2008_Step_quad} introduced the concept of Stepanov almost periodicity in mean-square. Their aim was to prove, under some conditions, existence and
uniqueness of Stepanov (quadratic-mean) almost periodic solution for a class
of nonautonomous stochastic evolution equations on a separable real Hilbert
space. This paper was the starting point of other works on stochastic differential equations with Stepanov-like ($\mu$-pseudo) almost periodic (automorphic) coefficients (see, e.g. \cite{chang_stepanov-like_2011,chang_stepanov-like_2016,yan_existence_2015,tang_stepanov-like_2014}). Unfortunately, the claimed results are erroneous \cite{bed-chal-mel-prf-sma2015,MRF13}.

The motivation of this paper has two sources. The first one comes from the papers by Andres and Pennequin \cite{Andres-Pennequin2012,andres_stepanov_2012}, who show the nonexistence of purely Stepanov-almost periodic solutions of ordinary
differential equations in uniformly convex Banach spaces. The second one comes from our paper \cite[Example~3.1]{bed-chal-mel-prf-sma2015}, where, with the simple counterexample of Ornstein-Uhlenbeck process, we have shown that even a one-dimensional linear equation with constant coefficients has no nontrivial solution which is Stepanov almost periodic in mean-square.

In this paper, we revisit the question of existence and uniqueness of Stepanov almost periodic solutions in both deterministic and stochastic cases. More precisely, we consider two semilinear stochastic evolution equations in a Hilbert space. The first one has  Stepanov almost periodic coefficients, and the second one has Stepanov $\mu$-pseudo almost periodic coefficients. We show that each equation has a unique mild solution which is almost periodic in $2$-UI distribution in the first case, and $\mu$-pseudo almost periodic in 2-UI distribution in the second case. Our results generalize and complete those of Da Prato and Tudor~\cite{DaPrato-Tudor95}, and those obtained recently by Kamenskii et al.~\cite{KMRF12averaging}, and Bedouhene et al.~\cite{bed-chal-mel-prf-sma2015}. We also show, by mean of a new superposition theorem in the deterministic case, the nonexistence of purely Stepanov $\mu$-pseudo almost periodic solutions to some evolution equations, generalizing a result of Andres and Pennequin \cite{Andres-Pennequin2012}.

The rest of the paper is organized as follows. In Section~\ref{sec:definition}, we investigate several notions of Stepanov almost periodicity in Lebesgue measure, and Stepanov ($\mu$-pseudo) almost periodicity for metric-valued functions. We see in particular that almost periodicity in Stepanov sense depends on the uniform structure of the state space. Special attention is paid to superposition operators between the spaces of Stepanov ($\mu$-pseudo) almost periodic metric-valued functions. Section~\ref{sec:SDE} is the main part of this paper. Therein, we study existence and uniqueness of bounded mild solutions to the abstract semilinear stochastic evolution equation on a Hilbert separable space
\begin{equation*}
dX(t)=AX(t)dt+F(t,X(t))dt+G(t,X(t))dW(t),
\end{equation*}%
where $F$ and $G$ are Stepanov ($\mu$-pseudo) almost periodic, satisfying Lipschitz and  growth conditions.  In the case of uniqueness, the solution can be ($\mu$-pseudo) almost periodic in $2$-UI distribution. Our approach is inspired from Kamenskii et al.~\cite{KMRF12averaging}, Da Prato and Tudor \cite{DaPrato-Tudor95}, and  Bedouhene et al.~\cite{bed-chal-mel-prf-sma2015}. The major difficulty is the treatment of the limits $F^\infty$ and $G^{\infty}$ provided by the Bochner criterion in Stepanov sense applied to $F$ and $G$  respectively. Thanks to an application of Komlós's theorem \cite{komlos67gene}, this difficulty dissipates by showing that  $F^\infty$ and $G^{\infty}$ inherit the same properties as $F$ and $G$  respectively.  Finally, Section~\ref{sec:conslusion} is devoted to some remarks and conclusions about the problem of existence of purely Stepanov almost periodic solutions. We show by a simple example in a one-dimensional setting, that one can obtain  bounded purely Stepanov almost periodic solutions when the forcing term is purely Stepanov almost periodic in Lebesgue measure.

\section{Stepanov Almost periodicity and its variants in metric space}\label{sec:definition}
In this section, we present the concept of Stepanov ($\mu$-pseudo) almost periodic function and related concepts like almost periodicity in Lebesgue measure. Moreover, we also recall some useful and key results. We begin with some notations.
\subsection{Notations}\label{sec:notations}
In what follows, $(\espE,d)$ is a complete metric space.
Unless otherwise stated, we keep the notation $d$ to designate the metric of any metric space $\espX$. When $\espX$ is a Banach space, its norm induced by $d$ will be denoted by $\norm{.}$.

Let $\espX$ and $\espY$ be two complete metric spaces, we denote some classical spaces as follows:

 \begin{itemize}
   \item $\Cont(\espX,\espY)$,   the space of continuous functions from $\espX$ to $\espY$;
   \item $\Cont_k(\espX,\espY)$,  the space $\Cont(\espX,\espY)$ endowed with the topology of
uniform convergence on compact subsets of $\espX$;
   \item $\BCont(\espX,\espY)$, the space of bounded continuous functions from $\espX$ to $\espY$;
   \item If $\espY$ is Banach space, we denote by $\CUB\bigl(\R,\espY\bigl)$ the space $\BCont(\R,\espY)$ endowed with the topology of uniform convergence on $\R$ whose norm is noted by $\norm{.}_\infty$.
 \end{itemize}

\subsection{Stepanov and Bohr almost periodicity}
Let us recall some definitions of (Stepanov) almost periodic functions and some key results.

\paragraph{Almost periodicity} Recall that a set $A\subset\R$ is {\em relatively dense} if there exists a real number $\ell>0$, such that $A\cap[a, a+\ell]\neq\emptyset$, for
all $a$ in $\R$.   We say that a continuous function $f:\R\rightarrow \espE$ is
{\em Bohr almost periodic} (or simply {\em  almost periodic}) if for all $\epsilon>0$, the set
 \begin{equation*}
  T(f,\epsilon):=\accol{\tau \in \R,\, \sup_{t\in \R} d \left( f(t), f(t+\tau) \right)<\varepsilon}
 \end{equation*}
 is relatively dense \cite{Amerio-Prouse-1971-book,Besicovitch-1955-book,Corduneanu89book}.
The numbers $\tau \in T(f,\epsilon) $ are called {\em $\epsilon$-almost periods}.
We denote by $\AP(\R,\espE)$ the space of Bohr almost periodic functions.
We have the following criteria for Bohr almost periodicity, established by Bochner \cite{bochner62new_approach} for complex-valued functions, see also \cite{DaPrato-Tudor95} for metric-valued functions:
\begin{theo} [\cite{bochner62new_approach,DaPrato-Tudor95}]Let $f:\R\rightarrow \espE$ be a continuous function. The following conditions are equivalent:
\begin{enumerate}
  \item $f\in \AP(\R,\espE)$.
  \item $f$ satisfies {\em Bochner criterion}, namely, the set $\{f(t+.),\,t\in \R\}$ is relatively compact in  $\BCont(\R,\espE)$ with respect to the uniform metric.
  \item \label{item:Bochner_double_criterion}For every pair of sequences $(\alpha'_n) \subset \R$  and $(\beta'_n) \subset \R$, one can extract common subsequences $ (\alpha_n) \subset (\alpha'_n)$ and $ (\beta_n) \subset (\beta'_n)$ such that \begin{equation}\label{def:Bochner double criterion}
 \lim_{n\rightarrow\infty}\lim_{m\rightarrow\infty}f(t+\alpha_n+\beta_m)=\lim_{n\rightarrow\infty}f(t+\alpha_n+\beta_n)
\end{equation}
pointwise.
\end{enumerate}
\end{theo}
If condition~\ref{item:Bochner_double_criterion} holds,  we say that $f$ satisfies {\em Bochner's double sequence criterion}. This criterion turns out to be a very powerful tool in applications to differential equations, see for instance \cite{DaPrato-Tudor95}.

\paragraph{Stepanov almost periodicity}
In all the sequel, unless stated otherwise, $p$ denotes a real number,
with $p\geq 1$.
Following
\cite{danilov_measure-valued_1997,danilov_almost_2000}, let
$M(\R,\espE)$ be the set of measurable functions from $\R$ to $\espE$
(we do not distinguish between functions that coincide almost
everywhere for Lebesgue's measure).
We fix a point $x_0$ in $\espE$. We denote by $\elp{p}(\R,\espE)$,  the subset of $M(\R,\espE)$ of locally $p$-integrable functions, that is, $$\elp{p}(\R,\espE)=\Big\{f\in M(\R,\espE),\,\mbox{ for any }a,b\in \R ;\,\int\limits_{[a,b]}d^{p}\left(
f(t),x_{0}\right) dt<+\infty\Big\}.$$
Define $\ellp{p}(0,1;\espE)$ as the set
$$\ellp{p}(0,1;\espE)=\Big\{f\in M(\R,E),\, \int\limits_{[0,1]}d^{p}\left(f(t),x_{0}\right) dt<+\infty\Big\}$$
which is a complete metric space, when it is endowed with the metric $$\D_{\ellp{p}}(f,g)=\CCO{\int\limits_{[0,1]}d^{p}(f(t),g(t)) dt}^{1/p}.$$
We denote by $\ellp{\infty}(\R,\espE)$ the space of all $\espE$-valued essentially bounded functions,
endowed with essential supremum metric.
Obviously, all the previous spaces do not depend on the choice of the point $x_0 \in \espE$.

We say that a locally $p$-integrable function
$f :\,\R\rightarrow\espE$ is {\em Stepanov almost periodic of order $p$} or
{\em $\St^p$-almost periodic},
if, for all $\varepsilon >0$, the set
\begin{equation*}
 \St^pT(f,\epsilon):=\Big\{\tau \in \R, \, D_{\St^p}^{d}\CCO{f\left( .+\tau \right) ,f\left( .\right)}
\leq \epsilon \Big\}
\end{equation*}
is relatively dense, where, for any locally $p$-integrable functions $f,g :\,\R\rightarrow\espE$,
\begin{equation*}
 D_{\St^p}^{d}(f ,g)=\sup_{x\in\R}
                 \CCO{\int_{x}^{x+1}d^p(f(t),g(t))\,dt}^{1/p}.
\end{equation*}
The space of $\St^p$-almost periodic $\espE$-valued functions
is denoted by $\StAP{p}(\R,\espE)$.
Let $\StAP{\infty}(\R,\espE)$ be the space of functions $f\in \ellp{\infty}(\R,\espE)$ such that for any
 $\epsilon >0$, there exists a relatively dense $T_{{\ellp{\infty}}}^{d}(f,\epsilon)$ such that
 \begin{equation*}
 \D_{\ellp{\infty}}^{d}\CCO{f\left( .+\tau \right) ,f\left( .\right)}
\leq \epsilon,\, \mbox{for all}\,\tau \in T_{{\ellp{\infty}}}^{d}(f,\epsilon).
\end{equation*}
 The relation $\lim_{p\rightarrow \infty} D_{\St^p}^{d}(f ,g)=D_{\ellp{\infty}}^{d}(f ,g)$
 holds for any functions $f,g\in M(\R,\espE)$, see \cite{Andres-Pennequin2012} for the proof.

As in the case of Bohr-almost periodic functions, we have similar characterizations of Stepanov almost periodic

functions. More precisely, let $f \in \elp{p}(\R,\espE)$, $p\in [1,+\infty[$. Then, the
 following statements are equivalent (compare with \cite[Theorem~1 and Proposition~6]{Hu-Mingarelli08}):
\begin{itemize}
 \item $f$ is $\St^p$-almost periodic.
\item $f$ is {\em $\St^p$-almost periodic in Bochner sense}, that is, from every real sequence $(\alpha'_n) \subset \R$ one can extract a subsequence
$(\alpha_n)$ of $(\alpha'_n)$ and there exists a function $g\in \elp{p}(\R,\espE)$ such that
$$\lim_{n\rightarrow +\infty} D_{\St^{p}}^{d}\left( f(.+\alpha_n),g(.)\right) = 0.$$
\item $f$ satisfies {\em Bochner's type double sequence criterion in Stepanov sense},
that is, for every pair of sequences
$\{\alpha^{'}_n\} \subset \R$ and $\{ \beta^{'}_n\} \subset \R$, there exist common subsequences
$(\alpha_n)\subset (\alpha'_n)$ and
  $(\beta_n)\subset (\beta'_n)$  such
  that, for every $t\in \R$, the limits
\begin{equation}\label{def:Bochner double sequence}
 \lim_{n\rightarrow\infty}\lim_{m\rightarrow\infty}f(t+\alpha_n+\beta_m)
 \text{ and }\lim_{n\rightarrow\infty}f(t+\alpha_n+\beta_n),
\end{equation}
exist and are equal, in the sense of the $\ellp{p}$-metric $$\D_{\ellp{p}}^{d}(h(t+.),g(t+.))=\CCO{\int_{[0,1]}d^p(h(t+s),g(t+s))\,ds}^{1/p}$$ for $h,g\in \elp{p}(\R,\espE)$.
\end{itemize}
 The equivalence between these statements was originally established in \cite{Hu-Mingarelli08} in the context of Banach spaces. However, one can provide a simpler proof of the equivalence of the three previous items based on the following observation that the concept of Stepanov almost periodicity can be seen as Bohr almost periodicity of some function with values in the Lebesgue space $\ellp{p}(0,1;\espE)$. More precisely, let $f^b$ denote the {\em Bochner transform} \cite{bochner33abstrakte} of a function $f \in \elp{p}(\R,\espE)$:
$$f^b :\,
\left\{ \begin{array}{lcl}
\R&\rightarrow&\espE^{[0,1]}\\
t&\mapsto & f(t+.).
\end{array}\right.
$$
Then $f \in \StAP{p}(\R,\espE)$ if, and only if, $f^b \in \AP(\R,\ellp{p}(0,1;\espE)$, and the previous equivalences become a simple consequence of the fact that $f_n \rightarrow f$ if and only if $f^b_n\rightarrow f^b$ (see e.g. \cite{Amerio-Prouse-1971-book,Andres_Bersani_Lesniak_01_APvarious_metric,Pankov-1990-book,Levitan_Zhikov82}).


 Since functions in $\StAP{p}(\R,\espE)$ are bounded with respect to the Stepanov metric, one denotes by $\St^p(\R,\espE)$  (or $\St^p(\R)$ when $\espE=\R$, and $\St(\R,\espE)$ when $p=1$) the set of all $D_{\St^p}^{d}$-bounded functions, that is, for some (or any) fixed $ x_{0}\in\espE$,
$$\St^p(\R,\espE)=\{f\in M(\R,\espE); D_{\St^p}^{d}(f,x_0) <+\infty\}.$$
So, from now on, the space $\StAP{p}(\R,\espE)$ will be seen as a (closed) subset of the complete metric space $(\St^p(\R,\espE),D_{\St^p}^{d})$. We have the following inclusions:
$$\AP(\R,\espE)\subset \StAP{\infty}(\R,\espE) \subset \StAP{p}(\R,\espE) \subset \StAP{q}(\R,\espE)\subset \StAP{1}(\R,\espE)\subset \St(\R,\espE) $$
for $p\geq q \geq 1$ and
$\AP(\R,\espE)=\StAP{p}(\R,\espE) \cap \C_u(\R,\espE)$, where $\C_u(\R,\espE)$ denotes the set of $\espE$-valued uniformly continuous
functions on $\R$.

For more properties and details about real and Banach-valued Stepanov almost periodic functions, we refer the reader for instance to the papers and monographs \cite{Amerio-Prouse-1971-book,Andres-Bersani-Grande-06-Hierarchy,Pankov-1990-book,Besicovitch-1955-book,Corduneanu89book,Fink74book,Levitan53,Levitan_Zhikov82}.

Beside the previous characterizations of the class
$\StAP{p}(\R,\espE)$, there is an other one based on the concept of
Stepanov almost periodicity in (Lebesgue) measure, invented by
Stepanov \cite{Stepanoff1926}.
This concept plays a significant role in the proof of our
superposition theorem in $\StAP{p}(\R,\espE)$ (Theorem
\ref{thm:theorem-composition}).
\paragraph{Stepanov almost periodicity in Lebesgue measure : the space $\StAP{0}$}
For any measurable set $A \subset \R$, let
$$\varkappa(A)=\sup_{\xi \in \R}\meas\CCO{[\xi,\xi+1]\cap A},$$ where
$\meas $ is the Lebesgue measure.
 A measurable function $f:\R\rightarrow \espE$ is said to be {\em Stepanov almost periodic in Lebesgue measure} or {\em $\St^0$-almost periodic} if for any $\epsilon, \, \delta>0$,  the set
\begin{equation*}
T_\varkappa(f,\epsilon,\delta):=\accol{\tau \in \R,\, \sup_{\xi \in \R}\meas{\accol{ t \in [\xi, \xi +1], d\bigl(f(t+\tau),f(t)\bigr) \geq \epsilon }}<
\delta}
\end{equation*}
is relatively dense. It should be mentioned that this almost periodicity was introduced as {\em $\mu$-almost periodicity} \cite{Stoinski_AP_Lebesgue_measure94}.
We denote by $\StAP{0}(\R,\espE)$ ($\StAP{0}(\R)$ when $\espE=\R$) the space of such fonctions.
This space was studied in depth by several authors (in both normed and metric spaces). One can quote Stoinski's works \cite{Stoinski_AP_Lebesgue_measure94,Stoinski99_compatness}, where an approximation property and  some compactness criterion are given. Danilov \cite{danilov_measure-valued_1997,Danilov-1997,danilov_almost_2000} has explored this class in the framework of almost periodic measure-valued functions. The recently published paper \cite{KASPRZAK_16}, that we discovered at the time of writing this paper, completes the previous ones. The authors of this paper investigate some other properties, in particular, they show that in general the mean value of $\St^0$-almost periodic functions may not exist, furthermore, $\St^0$-almost periodic functions are generally not Stepanov-bounded.

As pointed out by Danilov \cite{Danilov-1997}, $\St^0$-almost periodicity coincides with classical Stepanov almost periodicity when replacing the metric $d$ by $d'=\min(d,1)$. In other words, we have the following characterization (see \cite{Danilov-1997,Danilov98-uniform-approximation})
\begin{equation}\label{eq:caracterization S_D}
\StAP{0}(\R,\espE)=\StAP{1}\big(\R,(\espE,d')\big)=\StAP{p}\big(\R,(\espE,d')\big), \forall p>0.
\end{equation}
More generally, Stepanov almost periodicity can be seen as $\St^0$-almost periodicity under a uniform integrability condition in Stepanov sense (see \cite{danilov_measure-valued_1997,Franklin1929,KASPRZAK_16,Stepanoff1926}). To be more precise, let $M_p'(\R,\espE)$ be the set of $D_{\St^p}^{d}$-bounded functions such that
\begin{equation}
\lim_{\delta\rightarrow 0^+}\sup_{\xi \in \R}\underset{\meas{T}\leq \delta} {\sup_{T \subset [\xi, \xi +1]}}
\int_T d^p\bigl(f(t), x_0\bigr)dt=0.
\end{equation}
The space $M_p'(\R,\espE)$, $p\geq 1$,  is a closed subset of $(\St^p(\R,\espE),D_{\St^p}^{d})$.
In \cite[page 1420]{Danilov-1997}, Danilov gives an elegant characterization of
Stepanov almost periodic functions in terms of $M_p'(\R,\espE)$ and $\StAP{0}(\R,\espE)$, more precisely:
\begin{equation}\label{eq:caracterisation-Danilov}
 \StAP{p}(\R,\espE)=\StAP{0}(\R,\espE) \cap M_p'(\R,\espE).
\end{equation}

A rather interesting result about the space $\StAP{0}(\R,\espE)$
is reported in the following theorem \cite[Theorem $3$]{Danilov98-uniform-approximation}, which  gives a uniform approximation of Stepanov almost periodic functions by Bohr almost periodic functions, in the context of normed space $\espE$. Before, let us denote by $\cS(\R)$ the collection of measurable sets $T \subset \R$ such that the indicator function of $T$, $\un{T}$, belongs to $\StAP{1}(\R)$, and by $T^c$ the complementary set of $T$.
\begin{theo}[Danilov~\cite{Danilov98-uniform-approximation}]\label{thm:Danilov-compacite}
Let $f\in \StAP{0}(\mathbb{R},\espE)$, then for any $\delta>0$, there exist a set $T_{\delta}\in \cS(\R)$
and a Bohr almost periodic function $F_{\delta}$ such that
$\varkappa(T_{\delta}^c)<\delta$ and $f(t) = F_{\delta}(t)$ for all
$t \in T_{\delta}$.
\end{theo}
As consequence, we have the following corollary:
\begin{cor}\label{cor:Danilov-compacite}
Let $f\in \StAP{0}(\mathbb{R},\espE)$. Then, for all $\epsilon>0$, there exist a measurable set $T_{\epsilon} \in \cS(\R)$ and a compact subset $K_{\epsilon}$ of $\espE$ such that $\varkappa(T_{\epsilon}^c)<\epsilon$ and
$f(t)\in
K_{\epsilon}, \forall t\in T_{\epsilon}$.
\end{cor}
Danilov has shown that this property remains valid even in the metric framework \cite{danilov_almost_2000}.
\begin{remark}\label{rem:metric property of S_D}{\em
\begin{enumerate}
  \item Unlike almost periodicity in Bohr sense and  almost periodicity in Lebesque measure for function with values in a metric space $(\espE, d)$,
  which depend only on the topological structure of $\espE$ and not on its metric (see e.g., \cite{bedouhene-mellah-prf2012} and \cite{Danilov-1997} respectively), Stepanov almost periodicity is a metric property. In fact, as the metrics $d$ and $d'=\min(d,1)$ are topologically equivalent on $\espE$, we only need to show that the inclusion  $\StAP{1}\big(\R,(\espE,d)\big) \subset \StAP{1}\big(\R,(\espE,d')\big)$ is strict, since in view of \eqref{eq:caracterization S_D}, we have $\StAP{1}\big(\R,(\espE,d')\big)=\StAP{0}(\mathbb{R},\espE)$. Consider the example given in \cite[Remark~3.3]{Andres-Pennequin2012}. As shown by the authors, the function $g=\exp\CCO{\sum_{n=2}^{+\infty}g_n}$, where $g_n$ is the $4n$-periodic function given by
  $$g_n(t)=\beta_n\CCO{1-\frac{2}{\alpha_n}\abs{t-n}}
  \un{[n-\frac{2}{\alpha_n},n+\frac{2}{\alpha_n}]}(t),\quad
  t\in[-2n,2n],
$$
with $\alpha_n=1/n^5$ and $\beta_n=n^3$, is not in $\StAP{1}(\R)$. Using Danilov's Corollary \cite{Danilov-1997}, we get that $g$ belongs to $\StAP{0}(\mathbb{R})$, as a superposition of a continuous function and a periodic,
continuous and bounded function.
  \item Still in the spirit of the link between the spaces $\StAP{1}(\R)$ and $\StAP{0}(\R)$, an interesting property established by Stoi\'nski says that the inverse of any trigonometric polynomial with constant sign is $\St^0$-almost periodic. In particular, the Levitan function $f:\R\rightarrow \R$ given by $f(t)=\dfrac{1}{2+ \cos(t)+\cos(2t)}$ is $\St^{0}$-almost periodic but not Stepanov almost periodic (see Example~\ref{exple:Bohl-Bohr-Thm} and  \cite{KASPRZAK_16} for the second statement).
  \item Uniform integrability in Stepanov sense is a metric property, that is, the space $M_p'(\R,(\espE,d))$ depends on the metric $d$.
\end{enumerate}}
\end{remark}

\paragraph{Bohr and Stepanov almost periodic functions depending on a parameter} Hereinafter, some definitions of Bohr and Stepanov almost periodicity for metric-valued parametric functions are presented. Such definitions are simple adaptation of the well-known ones in the literature, see in particular \cite{bed-chal-mel-prf-sma2015,blot-al09superposition,diagana_2008_SPA,Ding-Long-N'guerekata2011,li_stepanov-like_2011}.
\begin{enumerate}
\item We say that a parametric  function $f :\,\R\times\espX\rightarrow \espY$ is {\em almost periodic with respect to the
first variable, uniformly with respect to the second variable in bounded subsets of $\espX$} (respectively {\em
in compact subsets of $\espX$}) if, for every  bounded (respectively compact) subset $B$ of $\espX$, the mapping $f:\R \rightarrow C(B,\espY)$ is almost periodic.
 We denote by $\APUb(\R\times\espX,\espY)$ and
$\APUc(\R\times\espX,\espY) $ respectively the spaces of
such functions.
 \item We say that a function $f:\R\times \espX\rightarrow \espY$ is $\St^{p}$-almost periodic if,
 for every $x\in \espX$, the $\espY$-valued function $f(.,x)$ is $\St^{p}$-almost periodic. We denote by
 $\StAP{p}\CCO{ \R \times \espX,\espY}$ the space of such functions.
 \item Let $f\in\StAP{p}\CCO{ \R \times \espX,\espY}$. We say that $f$ is $\St^{p}$-almost periodic uniformly with respect to the second variable in compact (resp. bounded) subsets of $\espX$ if $f(., x)$ is $\St^{p}$-almost periodic uniformly with respect to
 $x\in K$ for any compact (resp. bounded) subset $K$ of $\espX$. The space of such functions is denoted by
  $\StAPUc{p} \CCO{ \R \times \espX,\espY }$ (resp. $\StAPUb{p} \CCO{ \R \times \espX,\espY }$).
\end{enumerate}
   Clearly, we have the following inclusions: $$\StAPUb{p}(\R \times \espX,\espY)
\subset \StAPUc{p}(\R \times \espX,\espY)  \subset \StAP{p}(\R \times \espX,\espY) \subset \elp{p}(\R \times \espX,\espY),$$
where $\elp{p}(\R \times \espX,\espY)$ denotes
the set of measurable functions $f:\R\times \espX\rightarrow \espY$ such that, for all
$x\in \espX$; $f(.,x) \in \elp{p}(\R ,\espY)$.

The following proposition will be very useful in the sequel.
\begin{prop}\label{prop:stepanov-ponctuel}
  Let $\espY$ be a complete metric space, and let $\espX$ be a complete separable metric space. Let $f\in \StAP{p}(\R\times \espX, \espY)$
  satisfying the following Lipschitz condition:
  \begin{equation}\label{eq:lipschitze}
   d\CCO{f(t,x) , f(t,y)}\leq K(t) d\CCO{x,y}, \,\forall t \in \R,\, x,y \in \espX,
  \end{equation}
 for some positive function $K(.)$ in $\St^{p}(\R)$.
  Then for every real sequence $(\alpha_n^{'})$, there exist a subsequence $(\alpha_n) \subset (\alpha_n^{'})$ (independent of $x$) and a function
  $f^\infty \in \StAP{p}(\R\times \espX, \espY)$ such that for every
  $t\in \R$ and $x\in \espX$, we have
  \begin{equation}\label{eq:unicseq}
   \lim_n\int_t^{t+1}d\CCO{f(s+\alpha_n,x),f^\infty(s,x)}^{p}ds = 0.
  \end{equation}
 \end{prop}
\begin{proof}
 Firstly, let us show that $f^\infty$ is Lipschitz with respect to the second variable in the Stepanov metric sense. Let $x,y \in \espX$.
 We consider a real sequence $(\alpha_n^{'}) \subset \R$. Since $f\in \StAP{p}(\R\times \espX, \espY)$, for every $x \in \espX$,
 we can find a subsequence $(\alpha_n) \subset (\alpha_n^{'})$ (depending on $x$) such that
 \begin{equation}\label{eq:SAPX}
  \lim_nD_{\St^{p}}^{d}\CCO{f(.+\alpha_n,x) - f^\infty(.,x)}= 0.
 \end{equation}
For the same reason, for every $y \in \espX$, there exists a subsequence of  $(\alpha_n)$, (depending on both $x$ and $y$ still noted $(\alpha_n)$ for simplicity) such that
\begin{equation}\label{eq:SAPY}
  \lim_nD_{\St^{p}}^{d}\CCO{f(.+\alpha_n,y) , f^\infty(.,y)} = 0.
 \end{equation}
 Then, by  (\ref{eq:lipschitze}), (\ref{eq:SAPX}) and (\ref{eq:SAPY}), we get
 \begin{multline}\label{eq:conti}
 D_{\St^{p}}^{d}\CCO{f^\infty(.,x),f^\infty(.,y)} \leq  \lim D_{\St^{p}}^{d}\CCO{f^\infty(.,x),f(.+\alpha_n,x)} + \lim D_{\St^{p}}^{d}\CCO{f(.+\alpha_n,x),f(.+\alpha_n,y)}\\
                               +\lim D_{\St^{p}}^{d}\CCO{f(.+\alpha_n,y),f^\infty(.,y)}\\
                             \leq \|K\|_{\St^p}d{(x,y)}.
 \end{multline}
 Secondly, let us show (\ref{eq:unicseq}). Let $(\alpha_n^{'})$ be a real sequence. Since $\espX$ is separable, let $D$ be a dense countable subset of $\espX$.
 Using (\ref{eq:SAPX}) and a diagonal procedure, we can find a subsequence $(\alpha_n)$ of $(\alpha_n^{'})$ such that for every
 $t\in \R$ and $x\in D$, we have
  \begin{equation}\label{eq:count}
   \lim_n\int_t^{t+1}\CCO{d\CCO{f(s+\alpha_n,x) , f^\infty(s,x)}}^pds = 0
  \end{equation}
  Let $x\in \espX$, there exists a sequence $(x_k)\subset D$ such that $\lim_k d(x_k , x) = 0$. From \eqref{eq:lipschitze}, we deduce
  \begin{equation}\label{eq:convunif}
   \lim_k\int_t^{t+1}\CCO{d\CCO{f(s+\alpha_n,x_k), f(s+\alpha_n,x)}}^pds = 0,
  \end{equation}
  uniformly with respect to $n\in \N$. Now from \eqref{eq:count}, we obtain, for every $t\in \R$ and $k\in \N$,
  \begin{equation}\label{eq:convsimple}
   \lim_n\int_t^{t+1}\CCO{d\CCO{f(s+\alpha_n,x_k),f^\infty(s,x_k)}}^pds = 0.
  \end{equation}
  Using \eqref{eq:convunif}, \eqref{eq:convsimple} and by a classical result on interchange of limits, we deduce
  $$\lim_n f(s+\alpha_n,x) = \lim_kf^\infty(s,x_k)= f^\infty(s,x)$$in Stepanov metric. The last equality follows from \eqref{eq:conti}.
 \end{proof}
\finpr
\subsection{Bohr and Stepanov weighted pseudo almost periodic functions with values in metric space}
   The notions of Stepanov-like weighted pseudo almost periodicity and Stepanov pseudo almost periodicity of functions, with values in
Banach space $\espX$, were introduced by T. Diagana \cite{diagana06weighted,diagana07book,diagana_2008_SPA} as natural generalizations of the pseudo almost periodicity invented by Zhang \cite{zhang94,zhang95}.

Here we give the definitions of these different notions for functions with values in a complete metric space $\espE$. C. and M. Tudor \cite{Tudor-Tudor99} have proposed an elegant definition of pseudo almost periodicity in the context of metric spaces, which is slightly restrictive, since it requires compactness  of the range of the function instead of its boundedness. A more general definition of weighted pseudo almost periodicity (automorphic) has been introduced in \cite{bed-chal-mel-prf-sma2015}, where it is shown that there is no need to assume that $\espE$ is a vector space, nor a metric space, and these notions depend only on the topological structure of $\espE$. The definition we propose here (see Definition~\ref{def:PAP_metric}) is an intermediate between that of C. and M. Tudor and that in the wide sense \cite[Proposition~2.5 (i)]{bed-chal-mel-prf-sma2015}. It coincides with the one existing in the literature when $\espE$ is a normed space \cite{blot-cieutat_ezzinbi2012,blot-cieutat-ezzinbi2013}.

 We begin by recalling the definition of $\mu$-ergodicity for vector-valued functions \cite{blot-cieutat_ezzinbi2012,blot-cieutat-ezzinbi2013}. Let $(\espX; \norm{.})$ be a Banach space. Let $\mu$ be a Borel measure on $\R$ satisfying
\begin{equation}\label{eq:mu}
\mu(\R)=\infty \text{ and }
\mu(I)<\infty \text{ for every bounded interval $I$}.
\end{equation}
A function $f\in\BCont(\R,\espX)$ is said to be {\em$\mu$-ergodic} if it satisfies
\begin{equation*}
\lim_{r\rightarrow\infty}\frac{1}{\mu([-r,r])}\int_{[-r,r]}\norm{f(t)}\,d\mu(t)=0.
\end{equation*}
We denote by $\ER(\R,\espX,\mu)$ the space of {\em$\mu$-ergodic} $\espX$-valued functions. When $\mu$ is the Lebesgue measure, we merely denote this space by $\ER(\R,\espX)$. For more details about $\mu$-ergodic functions, we refer the reader to \cite{blot-cieutat_ezzinbi2012,blot-cieutat-ezzinbi2013}.
\begin{definition} \label{def:PAP_metric}{\em
A continuous and bounded function $f :\,\R\rightarrow\espE$ is said to be {\em pseudo almost periodic} if there exists a function $g\in\AP(\R,\espE)$ such that
the mapping $t\rightarrow d(f(t),g(t))$ is in $\ER(\R,\R)$, and it is said to be {\em $\mu$-pseudo almost periodic} or
{\em weighted pseudo almost periodic} if there exists a function $g\in\AP(\R,\espE)$ such that the mapping $t\rightarrow d(f(t),g(t))$ is in $\ER(\R,\R,\mu)$.}
\end{definition}
We denote respectively by $\PAP(\R,\espE)$ and $\PAP(\R,\espE,\mu)$ the spaces of such functions.
Note that $g$ is uniquely determined by $f$ in the first case (see \cite{Tudor-Tudor99}). This is not necessarily the case when considering $f\in\PAP(\R,\espE,\mu)$. However, it is easy to see that a sufficient condition of uniqueness of $g$ is that $\ER(\R,\espE,\mu)$ is translation invariant.
This is the case in particular if
Condition (\textbf{H}) of \cite{blot-cieutat_ezzinbi2012} is satisfied:
\begin{itemize}
\item[(\textbf{H})] For every $\tau\in\R$, there exist $\beta>0$ and a
  bounded interval $I$ such that $\mu(A+\tau)\leq\beta \mu(A)$
  whenever $A$ is a Borel subset of $\R$ such that $A\cap I=\emptyset$.
\end{itemize}

Let $p\geq 1$.
We use Bochner's transform to define the Stepanov $\mu$-pseudo almost periodicity for metric-valued functions:
\begin{definition}{\em
\begin{enumerate}
  \item We say that a locally $p$-integrable function
$f :\,\R\rightarrow\espE$ is {\em Stepanov $\mu$-ergodic} if $f^b\in\ER(\R,\ellp{p}([0,1],dt,\espE),\mu)$. Set
$$\StER{p}(\R,\espE,\mu)
=\accol{f\in \elp{p}(\R,\espE) \tq f^b\in\ER(\R,\ellp{p}([0,1],dt,\espE),\mu)}.$$
  \item We say that $f :\,\R\rightarrow\espE$ is

{\em $\St^p$-weighted pseudo almost periodic}, or
{\em Stepanov $\mu$-pseudo almost periodic}
if $$ f^b\in\PAP(\R,\ellp{p}([0,1];\espE),\mu),$$ that is, if there exists $g \in \StAP{p} (\R,\espE)$ such that
\begin{equation} 
\lim_{r\rightarrow\infty}\frac{1}{\mu([-r,r])}\int_{[-r,r]}\D_{\ellp{p}}\Big(f^b(t),g^b(t)\Big)\,d\mu(t)=0,
\end{equation}
or, equivalently,
\begin{equation} \label{eq:ergodicity_Stepanov}
\lim_{r\rightarrow\infty}\frac{1}{\mu([-r,r])}\int_{[-r,r]}\CCO{\int\limits_{0}^{1}d^p(f(t+s),g(t+s))ds}^{1/p}\,d\mu(t)=0.
\end{equation}
\end{enumerate}}
\end{definition}
We denote by $\StPAP{p}(\R,\espE,\mu)$ the space of such functions.  Note that, in this case, the function $g$ is uniquely determined if $\mu$ satisfies  Condition (\textbf{H}). In fact, assume that $g_{1}, g_{2}\in \StAP{p} (\R,\espE)$  define the same function $f$. Then, the mapping $t\rightarrow \D_{\ellp{p}}(g^{b}_{1}(t),g^{b}_{2}(t))$ is in $\ER(\R,\R,\mu) \cap \AP(\R,\R)$. It follows that $\D_{\ellp{p}}(g^{b}_{1}(t),g^{b}_{2}(t))=0$, for all $t\in \R$. Consequently $g_{1}=g_{2}, \, a.e.$.

We have the following characterization of $\StPAP{p}(\R,\espE,\mu)$:
\begin{prop}\label{rem: Blot-cieuta}
Eq. \eqref{eq:ergodicity_Stepanov} is equivalent  to\begin{equation} \label{eq:ergodicity_Stepanov_eq}
\lim_{r\rightarrow\infty}\frac{1}{\mu([-r,r])}\int_{[-r,r]}\int\limits_{0}^{1}d^p(f(t+s),g(t+s))ds\,d\mu(t)=0.
\end{equation}
\end{prop}

Our proof is inspired from  J. Blot and P. Cieutat \cite[Proposition 6.6]{blot-Cieutat16Completness}.

\proof The case when $p=1$ is obvious.
Set, for simplicity,  $d(f(t+s),g(t+s)):=h(t+s)$ and
\begin{equation}\label{eq:Stepanov_norm}
    \abs{H}_{\St^p(t)}:=\CCO{\int\limits_{0}^{1}\norm{H(t+s)}^pds}^{1/p}
\end{equation}
 for any $p$-locally integrable Banach-space valued function, $H$. Let us assume that $t\mapsto \abs{h}_{\St^p(t)}^p \in \ER(\R,\R,\mu)$, $p>1$. Using Hölder's inequality, we have
\begin{equation*}
\int_{[-r,r]}\abs{h}_{\St^p(t)}\,d\mu(t) \leq \CCO{\mu([-r,r])}^{1/q}\accol{\int_{[-r,r]}\abs{h}_{\St^p(t)}^p\,d\mu(t)}^{1/p}.
\end{equation*}
Thus,\begin{equation*}
\frac{1}{\mu([-r,r])}\int_{[-r,r]}\abs{h}_{\St^p(t)}\,d\mu(t)
\leq  \accol{\frac{1}{\mu([-r,r])}\int_{[-r,r]}\abs{h}_{\St^p(t)}^p\,d\mu(t)}^{1/p}
\end{equation*}
which leads to $\abs{h}_{\St^p(.)} \in \ER(\R,\R,\mu)$.
Conversely, set $M=\sup \limits_{t\in \R}\abs{h}_{\St^p(t)}^{p-1}<\infty$. We have
\begin{align*}
\frac{1}{\mu([-r,r])}\int_{[-r,r]}\abs{h}_{\St^p(t)}^p\,d\mu(t)
\leq & \frac{1}{\mu([-r,r])}\int_{[-r,r]}\abs{h}_{\St^p(t)}\abs{h}_{\St^p(t)}^{p-1}\,d\mu(t)\\
\leq & M\frac{1}{\mu([-r,r])}\int_{[-r,r]}\abs{h}_{\St^p(t)}\,d\mu(t)
\end{align*}
which means that $\abs{h}_{\St^p(.)}^p \in \ER(\R,\R,\mu).$

\finpr
From now on, we only deal with the weighted pseudo almost periodicity since pseudo almost periodicity is a special case.
\paragraph{Weighted pseudo almost periodic and Stepanov-like weighted pseudo almost periodic functions depending on a parameter}
Let $\mu$ be a Borel measure on $\R$ satisfying \eqref{eq:mu}. 
The definition of $\mu$-pseudo almost periodicity for functions with parameter we present here is just an adaptation of that presented in \cite{bed-chal-mel-prf-sma2015}. We recall it for the convenience of the reader. Let $\espX$ and  $\espY$ be two metric spaces.
\begin{itemize}
  \item
We say that  $f \in \BCont(\R\times\espX, \espY)$ is {\em $\mu$-pseudo almost periodic with respect to the first variable, uniformly
with respect to the second variable in compact subsets of $\espX$} if,
\begin{enumerate}
  \item for every $x\in\espX$,
$f(.,x)$ is $\mu$-pseudo almost periodic (in this case, we write $f\in \PAP(\R\times\espX,\espY,\mu)$),
  \item there exists  $g\in\APUc(\R\times\espX,\espY)$ such that:
$$\lim_{r\rightarrow \infty}1/\mu{([-r,r])}\int^{r}_{-r} d(f(t,x),g(t,x))\,d\mu(t)=0$$
uniformly with respect to $x$ in compact subsets of $\espX$.
\end{enumerate}
Denote by $\PAPUc(\R\times\espX,\espY,\mu)$ the collection of such functions.
\item  If $(\espY; \norm{.})$ is a Banach space, we say that a function $f \in \elp{p}(\R\times\espX,\espY)$ is {\em Stepanov $\mu$-ergodic with respect to the first variable, uniformly
with respect to the second variable in compact subsets of $\espX$} if,
\begin{enumerate}
  \item for every $x\in\espX$,
$f^b(.,x)\in \ER(\R,\espY,\mu)$,
  \item $f$ satisfies $$\lim_{r\rightarrow \infty}1/\mu{([-r,r])}\int^{r}_{-r} \abs{f(.,x)}_{\St^{p}(t)}\,d\mu(t)=0$$
uniformly with respect to $x$ in compact subsets of $\espX$, where $\abs{.}_{\St^{p}(t)}$ is defined by \eqref{eq:Stepanov_norm}.
\end{enumerate}
 The collection of such functions is denoted by $\StERUc{p}(\R\times\espX,\espY,\mu)$.
  \item A function $f \in \elp{p}(\R\times\espX,\espY)$ is said to be {\em Stepanov-like $\mu$-pseudo almost periodic with respect to the first variable, uniformly
with respect to the second variable in compact subsets of $\espX$} if there exists  $g\in\StAPUc{p}(\R\times\espX,\espY)$ such that
\begin{equation*}
\left[(t,x)\mapsto d\Big(f(t,x),g(t,x)\Big)\right]\in \StERUc{p}(\R \times \espX,\R,\mu).
\end{equation*}
The space of such functions is denoted by $\StPAPUc{p}(\R\times\espX,\espY,\mu)$.
\end{itemize}
\subsection{A superposition theorem in $\StPAP{p}(\R,\espE,\mu)$}\label{subsec:superposition}
In this section, we study some properties of parametric functions, especially Nemytskii's operators $\mathcal{N}(f)(x):=[t\mapsto f(t,x(t))]$  built on $f:\R \times \espE\rightarrow \espE$ in the space of Stepanov ($\mu$-pseudo) almost periodic functions.
In the following, we assume that $\frac{1}{p}=\frac{1}{q}+\frac{1}{r}$ with $p$, $q$ and $r \geq 1$, and we consider the parametric function $f:\R \times \espE \rightarrow \espE$, which satisfies the Lipschitz condition:
\begin{itemize}
\item[(Lip)] There exists a nonnegative function $L\in \St^r(\R)$ such that
$$d(f(t,u),f(t,v))\leq L(t) d(u,v), \,\, \forall t\in \R, \,\,u,v\in \espE.$$
\end{itemize}
Using compactness property of Stepanov almost periodicity given by Danilov (see Theorem~\ref{thm:Danilov-compacite} and Corollary~\ref{cor:Danilov-compacite}), we improve the composition theorem of Stepanov almost periodic functions given in
\cite[Theorem 2.1]{Ding-Long-N'guerekata2011}. We show in particular that in order to obtain that $\mathcal{N}(f)$ maps $\StAP{q}(\R,\espE)$ into $\StAP{p}(\R,\espE)$, the following compactness condition:
\begin{itemize}
\item[(Com)] There exists a subset $A \subset \R$ with $\meas(A)=0$ such that
$K:=\overline{\{x(t): t \in \R\backslash A\}}$ is a compact subset of $\espE$,
\end{itemize}
is not necessary.
Let us mention that Andres and Pennequin \cite[Lemma~3.2]{Andres-Pennequin2012} have shown that, given a Banach space $\espX$,  a continuous function $f:\espX\rightarrow \espX$ satisfying, for some $a,b>0$ and $\,p,q\geq1$, the following growth condition:$$\forall x \in \espX,\, \norm{f(x)} \leq a \norm{x}^{p/q}+b,$$ and an $\StAP{p}$-function $g:\R\rightarrow \espX$, the composition  $f\circ g$ is an  $\StAP{q}$-function.

We begin by a lemma which identifies the spaces $\StAP{p}(\R\times \espE,\espE)$ and $\StAPUc{p}(\R\times \espE,\espE)$ under Condition (Lip).
\begin{lemma} \label{lem:SPP-uniform-compact}
Let
$f:\R \times \espE \rightarrow \espE$ be a parametric function satisfying Condition $\lip$. Then,
$f\in \StAP{p}(\R\times \espE,\espE)$ if, and only if, $f\in \StAPUc{p}(\R\times \espE,\espE)$.
\end{lemma}
The proof of this lemma is very similar to that
of Fan et al. \cite[Lemma~3.1]{Fan-Liang-Xiao2011}
which is the analogous result
for almost automorphic case.

Now, before giving the superposition theorem in $\StAP{p}(\R\times\espE ,\espE)$, we need some more notations. Let $u:\R\rightarrow\espE$ be a mesurable function, let $A\subset \R$ be a measurable set. We  denote by $u^{\Trunc{A}}$ the "truncated" function  from $\R$ to $\espE$, defined by $$ u^{\Trunc{A}}(t)=
\left\{ \begin{array}{lcl}
u(t)& \mbox{if} & t\in A\\
x_{0}&\mbox{if} & t\notin A.
\end{array}\right.
$$

We are now ready to present the superposition theorem in $\StAP{p}(\R\times\espE ,\espE)$:
\begin{theo}\label{thm:theorem-composition}
Let  $f \in \StAP{p}(\R\times\espE ,\espE)$, and assume that $f$ satisfies Condition {\lip}.
Then, for every  $u\in \StAP{q}(\R ,\espE)$, we have $f(.,u(.))\in \StAP{p}(\R ,\espE)$.
\end{theo}
\proof{
Fix $\epsilon >0$ and $x_0 \in \espE$. Let $u\in \StAP{q}(\R ,\espE)$.
In view of \eqref{eq:caracterisation-Danilov}, we have $u\in M_p'(\R,\espE)\cap \StAP{0}(\R,\espE)$, thus, there exists $\eta >0$ such that
$D_{\St^q}^d \big(u^{\Trunc{A}}(.),x_0\big)\leq \epsilon
$ for all measurable set $A\subset \R$ satisfying $\varkappa(A) \leq \eta$. For such $\eta$, using Corollary \ref{cor:Danilov-compacite}, we deduce that there exists a compact subset $\mathcal{K}_{\eta (\epsilon ) }\subset \espE$ such that
\[
\varkappa \{ t\in \R,\,  \ x(t)\notin \mathcal{K}_{\eta(\epsilon)} \} <\eta
\]%
and
\begin{equation}\label{eq:UI x supperposition thm}
  D_{\St^q}^d \bigr(u^{\Trunc{T_{ \epsilon}^{c}}}(.),x_0\bigl)
        \leq \frac{\epsilon}{6 \norm{L}_{\St^r}},
      \end{equation}
      where $T_\epsilon:=T_{\eta (\epsilon)}$ is the subset of $\R$ on
      which $u(t)\in\mathcal{ K}_{\eta( \epsilon ) }$
      (we exclude for simplicity the trivial case when  $\norm{L}_{\St^r}=0$).
The compactness of $\mathcal{K}_{\eta(\epsilon )}$ implies that there
exists a finite sequence $(x_{1}, x_{2},...,x_{n})$ in $\mathcal{K}_{\eta(\epsilon )}$ such that
\begin{equation}\label{eq:compactness supperposition thm_SAP}
\mathcal{K}_{\eta(\epsilon )} \subset \bigcup \limits_{1\leq i\leq
n}B\CCO{ x_{i},\frac{\epsilon }{6\norm{L}_{\St^r}}}.
\end{equation}
For $i=1,\dots,n$ and $t\in \R$, let
$$\ii(t)=
\begin{cases}
0&\text{ if }t \not\in T_\epsilon\\
\min\left\{i\in\{1,\dots,n\} \tq
  d(u(t),x_i)\leq\frac{\epsilon}{ 6\norm{L}_{\St^r} }\right\}
        &\text{ if }t \in T_\epsilon,
\end{cases}$$
and, for $i=0,\dots,n$ and $\xi\in\R$, let
$$A_{i,\xi}=\{t\in[\xi,\xi+1] \tq \ii(t)=i\}.$$

By Lemma \ref{lem:SPP-uniform-compact}, we have $f \in
\StAPUc{p}(\R\times \espE,\espE)$. Since $u\in \StAP{q}(\R ,\espE)$,
we can choose a common relatively dense set
$\mathcal{T}(f,u,\epsilon)\subset \R$ such that,
for $\tau \in \mathcal{T}(f,u,\epsilon)$,
\begin{equation} \label{eq:thm-com-SPP-x}
 D_{\St^q}^d (u(.+\tau),u(.)) \leq \frac{\epsilon}{3\norm{L}_{\St^r}}
\end{equation}
and
\begin{equation} \label{eq:thm-com-SPP-f}
 \sum_{i=0}^{n}D_{\St^p}^d \Big( f(.+\tau ,x_i),f(. ,x_{i})\Big)
    \leq \frac{\epsilon}{3}
\end{equation}
for all $\tau \in \mathcal{T}(f,u,\epsilon)$.
Let $\tau \in \mathcal{T}(f,u,\epsilon)$. We have
\begin{multline*}
D_{\St^p}^d \Bigl(f\CCO{.+\tau,u(.+\tau)},f(.,u(.))\Bigr)\\
\begin{aligned}
  \leq & D_{\St^p}^d \Bigl(f(.+\tau,u(.+\tau)),f(.+\tau,u(.))\Bigr)
  +D_{\St^p}^d \Bigl(f(.+\tau,u(.)),f(.,u(.))\Big) \\
  \leq & \norm{L}_{\St^r}D_{\St^q}^d \Big(u(.+\tau),u(.)\Bigr)
         + D_{\St^p}^d \Bigl(f(.+\tau,u(.)),f(.,u(.))\Bigr) \\
\leq & \frac{\epsilon}{3}+D_{\St^p}^d \Bigl(f(.+\tau,u(.)),f(.,u(.))\Bigr).
\end{aligned}
\end{multline*}
Now,
\begin{align*}
 & D_{\St^p}^d \Big(f(.+\tau,u(.)),f(.,u(.))\Big)
  =\sup_{\xi\in\R}\CCO{\int_\xi^{\xi+1}d^p(f(t+\tau,u(t)),f(t,u(t)))\,dt}^{1/p}\\
  \leq& \sup_{\xi\in\R}\CCO{\int_\xi^{\xi+1}
    \sum_{i=0}^{n}\un{A_i,\xi}(t)d^p(f(t+\tau,u(t)),f(t+\tau,x_i))\,dt}^{1/p}\\
      &+\sup_{\xi\in\R}\CCO{\int_\xi^{\xi+1}
      \sum_{i=0}^{n}\un{A_i,\xi}(t)d^p(f(t+\tau,x_i),f(t,x_i))\,dt}^{1/p}\\
        &+\sup_{\xi\in\R}\CCO{\int_\xi^{\xi+1}
      \sum_{i=0}^{n}\un{A_i,\xi}(t)d^p(f(t,x_i),f(t,u(t)))\,dt}^{1/p}\\
  \leq & \sup_{\xi\in\R}\CCO{\int_\xi^{\xi+1}
         L(t+\tau)\sum_{i=0}^{n}\un{A_i,\xi}(t)d^p(u(t),x_i)\,dt}^{1/p}\\
   &+\sum_{i=0}^{n}D_{\St^p}^d \Big( f(.+\tau ,x_i),f(. ,x_{i})\Big)\\
   &+\sup_{\xi\in\R}\CCO{\int_\xi^{\xi+1}
         L(t)\sum_{i=0}^{n}\un{A_i,\xi}(t)d^p(u(t),x_i)\,dt}^{1/p}\\
\leq &\sup_{\xi\in\R}\CCO{\int_\xi^{\xi+1}L^r(t+\tau)\,dt}^{1/r}
 \Biggl\lgroup
            D_{\St^q}^d \bigr(u^{\Trunc{T_{
                  \epsilon}^{c}}}(.),x_0\bigl)
            +\CCO{\int_\xi^{\xi+1}
              \sum_{i=1}^{n}\un{A_i,\xi}(t)d^q(u(t),x_i)\,dt}^{1/q}
          \Biggr\rgroup\\
       &+ \frac{\epsilon}{3}
        +\sup_{\xi\in\R}\CCO{\int_\xi^{\xi+1}L^r(t)\,dt}^{1/r}
        \Biggl\lgroup
        D_{\St^q}^d \bigr(u^{\Trunc{T_{ \epsilon}^{c}}}(.),x_0\bigl)
      + \CCO{ \int_\xi^{\xi+1}
         \sum_{i=1}^{n}\un{A_{i,\xi}}(t)d^q(u(t),x_i)\,dt }^{1/q}
                  \Biggr\rgroup\\
 \leq &2\norm{L}_{\St^r} \frac{\epsilon}{ 6\norm{L}_{\St^r}}+\frac{\epsilon}{3}
 =\frac{2\epsilon}{3}.
\end{align*}
Combining the preceding inequalities, we deduce that
$$D_{\St^p}^d \Bigl(f(.+\tau,u(.+\tau)),f(.,u(.))\Bigr)
\leq \frac{\epsilon}{3}+\frac{2\epsilon}{3}=\epsilon.$$
\finpr

\begin{theo}\label{thm:superposition_SPAP}
Let $\mu$ be a Borel measure on $\R$ satisfying \eqref{eq:mu} and Condition $(\rm{H})$, and let $p\geq 1$. Assume that $F\in  \StPAPUc{p}(\R \times \espE,\espE,\mu)$. Let $G$ be in $\StAPUc{p}(\R \times \espE,\espE)$ such that
\begin{equation}
\left[(t,x)\mapsto H(t,x)=d\Big(F(t,x),G(t,x)\Big)\right]\in \StERUc{p}(\R \times \espE,\R,\mu).
\end{equation}
Assume that Condition {\lip}  holds for $F$ and $G$. If $X$ is $\St^q$-weighted pseudo almost periodic, then {\em $F(.,X(.)) \in \StPAP{p}(\R,\espE,\mu)$}.
\end{theo}
\proof Since  $X \in \StPAP{q}(\R,\espE,\mu)$, there exists  $Y\in \StAP{q}(\R ,\espE)$ such that
\begin{equation}\label{eq:thm super pseudo Z}
 Z(.)=d\Big(X(.),Y(.)\Big) \in \StER{p}(\R,\R,\mu).
\end{equation}
To show that $F(.,X(.)) \in \StPAP{p}(\R,\espE,\mu)$, it is enough to have $$d\Big(F(.,X(.)),G(.,Y(.))\Big)\in \StER{p}(\R,\R,\mu),$$ since by Theorem~\ref{thm:theorem-composition}, the function $G(.,Y(.)) \in \StAP{p}(\R,\espE)$.
We have $$d\Big(F(t,X(t)),G(t,Y(t))\Big)\leq d\Big(F(t,Y(t)),G(t,Y(t))\Big)+d\Big(F(t,Y(t)),F(t,X(t))\Big).$$ Clearly, $d\Big(F(.,Y(.)),F(.,X(.)\Big) \in \StER{p}(\R,\R,\mu)$. Indeed, for every $r>0$,
\begin{equation*}
\frac{1}{\mu{([-r,r])}}\int^{r}_{-r}
  \D_{\ellp{p}}^d\Big(F^b(t,Y^b(t)),F^b(t,X^b(t)\Big)\,d\mu(t)
\leq  \frac{1}{\mu{([-r,r])}}\int^{r}_{-r}\norm{L}_{\St^r}\D_{\ellp{q}}^d \Big(Y^b(t),X^b(t)\Big)\,d\mu(t)\\
\end{equation*}
and hence, using \eqref{eq:thm super pseudo Z}, it follows that
$$\lim_{r\rightarrow \infty}\frac{1}{\mu{([-r,r])}}\int^{r}_{-r}
  \D_{\ellp{p}}^d\Big(F^b(t,Y^b(t)),F^b(t,X^b(t)\Big)\,d\mu(t)=0.$$
  Now, we claim that $d\Big(F(.,Y(.)),G(.,Y(.))\Big)\in \StER{p}(\R,\R,\mu)$. In fact, let $\epsilon>0$ and $x_0 \in \espE$. Since $Y\in \StAP{p}(\R,\espE)\subset M_p'(\R,\espE)$,  there exists $\delta:=\delta(\epsilon) >0$ such that
$D_{\St^q}^d (Y^{\Trunc{A}}( .) ,x_0)\leq \epsilon
$ for all measurable set $A\subset \R$ satisfying $\varkappa(A) \leq \delta$. Thus, using Corollary \ref{cor:Danilov-compacite}, we deduce that there exists a compact subset $\mathcal{K}_{\delta } \subset \espE$ such that
\[
\varkappa \{ t\in \R,\,  \ Y(t)\notin \mathcal{K}_{\delta} \} <\delta
\]%
and
\begin{equation}\label{eq:UI x supperposition pseudo thm}
D_{\St^q}^d (Y^{\Trunc{T_{\delta}^c}}( .) ,x_0)\leq \frac{\epsilon}{4 \norm{L}_{\St^r}}
\end{equation}
where $T_\delta:=T_{\delta (\epsilon)}$ is the subset of $\R$ on which $Y(t)\in\mathcal{ K}_{\delta }$.
The compactness of $\mathcal{K}_{\delta(\epsilon )}$ implies that there exist $y_{1}, y_{2},...,y_{m} \in \mathcal{K}_{\epsilon}$ such that
\begin{equation}\label{eq:compactness supperposition thm}
\mathcal{K}_{\epsilon } \subset \bigcup \limits_{1\leq i\leq
m}B\CCO{ y_{i},\frac{\epsilon }{4\norm{L}_{\St^r}}}.
\end{equation}
Hence,
  \begin{multline*}
\frac{1}{\mu{([-r,r])}}\int^{r}_{-r}
  \D_{\ellp{p}}^d\Big(F^b(t,Y^b(t)),G^b(t,Y^b(t)\Big)\,d\mu(t)\\
\begin{aligned}
\leq & \max_{1\leq i\leq m}\frac{1}{\mu{([-r,r])}}\int^{r}_{-r}\D_{\ellp{p}}^d\Big(F^b(t,Y^b(t)),F^b(t,y_i)\Big)\,d\mu(t)\\
+ & \max_{1\leq i\leq m}\frac{1}{\mu{([-r,r])}}\int^{r}_{-r}\D_{\ellp{p}}^d\Big(F^b(t,y_i),G^b(t,y_i)\Big)\,d\mu(t)\\
+ & \max_{1\leq i\leq m}\frac{1}{\mu{([-r,r])}}\int^{r}_{-r}\D_{\ellp{p}}^d\Big(G^b(t,y_i),G^b(t,Y^b(t))\Big)\,d\mu(t)\\
:=&J_1(r)+J_2(r)+J_3(r).
\end{aligned}
\end{multline*}
We have from \eqref{eq:UI x supperposition pseudo thm} and \eqref{eq:compactness supperposition thm}:
\begin{align*}
J_1(r)+J_3(r)\leq& 2\norm{L}_{\St^r} \max_{1\leq i\leq m} D_{\St^q}^d\Big(Y(.),y_i\Big)\\
\leq & 2 \norm{L}_{\St^r}\max_{1\leq i\leq m}\accol{D_{\St^q}^d \Big(Y^{\lfloor T_{\delta},y_{i} \rfloor}( .),y_i\Big)+D_{\St^q}^d \Big(Y^{\lfloor T_{\delta}^{c},y_{i}\rfloor}( .),y_i\Big)}\leq \epsilon.
\end{align*}
Let us estimate $J_2(r)$. Using the fact that the parametric function $H$ is in $\StERUc{p}(\R\times \espE,\R,\mu)$, we have
$$J_2(r)\leq \sup_{y\in \mathcal{K}_\epsilon}\frac{1}{\mu{([-r,r])}}\int^{r}_{-r}\D_{\ellp{p}}^d\Big(F^b(t,y),G^b(t,y)\Big)\,d\mu(t),$$ from which we deduce that $\lim_{r\rightarrow \infty}J_2(r)=0$.
Finally, since $\epsilon$ is arbitrary, we obtain that
$$\lim_{r\rightarrow \infty}\frac{1}{\mu{([-r,r])}}\int^{r}_{-r}
  \D_{\ellp{p}}^d\Big(F^b(t,Y^b(t)),G^b(t,Y^b(t)\Big)\,d\mu(t)=0.$$
  This proves that $F(.,X(.))\in\StPAP{p}(\R,\espE,\mu)$.
\finpr

\subsection{Weighted pseudo almost periodicity
for stochastic processes}
\label{sec:stochproc}
This subsection is devoted to some definitions related to $\mu$-pseudo almost periodicity for a stochastic process. As observed in \cite{bedouhene-mellah-prf2012}, \cite{bed-chal-mel-prf-sma2015} and \cite{Tudor95ap_processes}, there are various modes and extension of almost periodicity for a stochastic process. We recall here the more relevant for applications to stochastic differential equations, namely, the $\mu$-pseudo almost periodicity in $p$-UI distribution, $p\geq 0$, proposed in \cite{bed-chal-mel-prf-sma2015}.  For the convenience of the reader, we repeat some notations and definitions from \cite{bed-chal-mel-prf-sma2015} and \cite{DaPrato-Tudor95}, thus making our exposition self-contained. Let $(\espX,\norm{.})$ be a separable Banach space, and let $(\esprob,\tribu,\prob)$ be a probability space. For a random variable $X : (\esprob,\tribu,\prob)\rightarrow \espX$, we denote by $\law{{X}}$ its law (or distribution) and by $\expect (X)$ its expectation. The space of all random variables from $\esprob$ to $\espX$ is denoted by $\ellp{0}(\esprob,\prob,\espX)$.  Note that $\CCO{\ellp{0}(\esprob,\prob,\espX),d_{\text{\tiny Prob}}}$, where $d_{\text{\tiny Prob}}$ is the distance that generates the topology of
convergence in probability,
$$d_{\text{\tiny Prob}}(U,V)=\expect \CCO{\norm{U-V}\wedge1},$$
for $U,\,V \in \ellp{0}(\esprob,\prob,\espX)$, is complete. For $p\geq 1$, let $\ellp{p}(\esprob,\prob,\espX)$ stand for the space of all $\espX$-valued random variables, $X$,  such that $\expect \norm{X}^p:=\int_\Omega \norm{X}^p d\prob <+\infty$. We equip this space with its natural norm that we denote by $\norm{.}_p:=\CCO{\expect \norm{X}^p}^{1/p}$.

Let $\laws{\espX}$  be the set of Borel probability measures on $\espX$. We endowed $\laws{\espX}$ with tow topologies. The first one is that }of narrow (or weak) convergence. For a given $\varphi \in \BCont(\espX,\R)$ and $\upsilon,\nu \in \laws{\espX}$, we define
\begin{align*}
\norm{\varphi}_{\lr}
&= \sup\Bigl\{\dfrac{\varphi(x)-\varphi(y)}{d(x,y)} \tq x \neq y\Bigl\}\\
\norm{\varphi}_{\bl} &= \max\{\norm{\varphi}_\infty,\norm{\varphi}_{\lr}\}\\
d_{\bl}(\upsilon, \nu) &= \sup_{\norm{\varphi}_{\bl}\leq 1}
\int_\espX \varphi \,d(\upsilon-\nu).
\end{align*}
The space $\CCO{\laws{\espX},d_{\bl}}$ is a Polish space and $d_{\bl}$ generates the weak topology on $\laws{\espX}$.  An other useful topology related to the metric space $\laws{\espX}$ is that induced by the convergence in Wasserstein distance. Let $p\geq 1$. For any $\upsilon, \nu \in \laws{\espX}$, the Wasserstein distance of order $p$, $\WASS^p$,
is defined by
\begin{equation}\label{eq:wass2}
\WASS^p(\upsilon, \nu)
=\inf\accol{ \CCO{\expect \bigl(\norm{X-Y}^p\bigr)}^{1/p}, \law{X}=\upsilon, \, \law{Y}=\nu}
\end{equation}
where the notation means that the infimum (actually, a minimum)
is taken over all pairs
$(X,Y)$ of $\espX$-valued random variables defined on some probability
space (not necessarily $(\esprob,\tribu,\prob)$) such that
$\law{X}=\upsilon$ and $\law{Y}=\nu$. For any interval $[a,b]$, we denote by $\WASS^p_{[a,b]}$ the Wasserstein distance between the distributions of two continuous $\espX$-valued stochastic processes $X=(X_t)_{t\in\R}$ and $Y=(Y_t)_{t\in\R}$, viewed as $\Cont([a, b],\espX)$-valued random variables. Let $(X_n)\subset \ellp{p}(\esprob,\prob,\espX)$ be a sequence of random vectors of $\espX$ and let $X \in \ellp{p}(\esprob,\prob,\espX)$. The convergence in Wasserstein distance can be characterized as follows (see e.g.~\cite{villani09oldnew}): The sequence $\CCO{\law{X_n}}$ converges to $\law{X}$ for $\WASS^p$ if and only if
\begin{enumerate}
  \item $(X_n)$ converges to $X$ in distribution, i.e. $d_{\bl}(\law{X_n}, \law{X})\rightarrow 0$ as $n\rightarrow \infty$;
  \item the family $(\norm{X_n}^p)$ is uniformly integrable ($p$-UI).
\end{enumerate}
 Let us mention that $\CCO{\laws{\espX},\WASS^p}$ is also a Polish space.

Now, in connection with the results of Section~\ref{sec:SDE}, let us recall the following definitions for a stochastic process  $X=(X_t)_{t\in\R}$ with
values in $\espX$, defined on the probability space $(\esprob,\tribu,\prob)$.
 \begin{definition}[\cite{Tudor95ap_processes}]{\em
\begin{enumerate}
  \item We say that $X$ is {\em almost periodic in one-dimensional
  distributions},  and write $X\in \APOD(\R,\espX)$, if the mapping $$\law{{X}(.)}:\left\{
                                    \begin{array}{lll}
                                      \R & \rightarrow & \laws{\espX} \\
                                      t & \mapsto & \law{{X}(t)}
                                    \end{array}
                                  \right.$$ is almost periodic.
  \item If $X$ has continuous trajectories, we say that $X$ is {\em almost periodic in distribution}, and write $X\in \APD(\R,\espX)$,
if the mapping  $$\law{\transl{X}(.)}:\left\{
                                    \begin{array}{lll}
                                      \R & \rightarrow & \laws{\Cont_k(\R,\espX)} \\
                                      t & \mapsto & \law{\transl{X}(t)}
                                    \end{array}
                                  \right.$$
 is almost periodic, where $\transl{X}(t)$ stands for the random variable $X(t+.)$ with values in $\Cont(\R,\espX)$.
\end{enumerate}}
\end{definition}
One can define, in a same manner, the almost periodicity in $\CCO{\laws{\espX},\WASS^p}$ and $\CCO{\laws{\Cont_k(\R,\espX)},\WASS^p_{[a,b]}}$. We denote by $\WASS^p\!\APOD(\R,\espX)$ and $\WASS^p_{[a,b]}\!\APD(\R,\espX)$ the classes obtained respectively.

Let us also recall the definition of  almost periodicity in
$p$-UI distribution\footnote{ This notion was introduced in \cite{bed-chal-mel-prf-sma2015} under the name "almost periodic (automorphic) in $p$-distribution". We have modified the terminology, and introduced the symbol "$p$-UI" defined above to bring out from the notion "$p$-distribution" the condition of "$p$-uniform integrability". This seems clearer and more evocative.}  introduced in
  \cite{bed-chal-mel-prf-sma2015} under a different name, which requires uniform integrability property:
\begin{definition}[\cite{bed-chal-mel-prf-sma2015}]{\em
Let $p\geq 0$. An $\espX$-valued stochastic process, $X$, is called
 {\em almost periodic in $p$-UI one-dimensional distribution} (resp. {\em almost periodic in $p$-UI distribution, if $X$ is continuous}) if
\begin{enumerate}[(i)]
\item $X\in\APOD(\R,\espX)$ (resp. $X\in\APD(\R,\espX)$);
\item if $p>0$, the family $(\norm{X(t)}^p)_{t\in\R}$ is uniformly integrable ($p$-UI).
\end{enumerate}}
\end{definition}
We denote by  $\APOD^p(\R,\espX)$ and $\APD^p(\R,\espX)$ the set of $\espX$-valued processes
which are almost periodic in $p$-UI one-dimensional distribution and in $p$-UI distribution, respectively. Note that when $p=0$, one obtains $\APOD^0(\R,\espX)=\APOD(\R,\espX)$ and $\APD^0(\R,\espX)=\APD(\R,\espX)$.
It should be mentioned that if $X\in \APD^p(\R,\espX)$, the mapping $t\mapsto X(t)$, $\R\rightarrow\ellp{p}(\esprob,\prob,\espX)$,
is continuous.
  \begin{remark}{\em
Let us clarify the relation between the classes $\APOD^p(\R,\espX)$, $\APD^p(\R,\espX)$ and the corresponding almost periodicity in the Wasserstein distance sense. We use Bochner's double sequences criterion. Let  $(\alpha'_n)\subset \R$ and $(\beta'_n)\subset \R$ be arbitrary sequences. Let $(\alpha_n) \subset (\alpha'_n)$ and $(\beta_n)\subset (\beta'_n)$ be the subsequences provided by Bochner's double sequences criterion. We just need to compare the corresponding UI. Our reasoning is based on the following observation that a family is UI if, and only if, from any sequence extracted from this family, one can extract a subsequence that is UI.
\begin{enumerate}
  \item Observe that if $X :\R\rightarrow \ellp{p}(\esprob,\prob,\espX)$, $p>0$, is bounded, we have $$X\in \APOD^p(\R,\espX) \Leftrightarrow X\in \WASS^p\!\APOD(\R,\espX).$$  Indeed, as the sequences $(\alpha'_n)$ and $(\beta'_n)$ are arbitrary, it is easy to check that the family $(\norm{X(t)}^p)_{t\in\R}$ is UI if, and only if, for any $t\in \R$, the sequences $(\norm{X(t+\alpha_n+\beta_m)}^p)_{n,m}$ and $(\norm{X(t+\alpha_n+\beta_n)}^p)_{n}$ are too.
  \item Let $X \in \CUB\bigl(\R,\ellp{p}(\esprob,\prob,\espX)\bigl)$. Then, $$ X \in \WASS^p_{[a,b]}\!\APD(\R,\espX)\Rightarrow X\in  \APD^p(\R,\espX).$$
      For, let $(s'_n)\subset \R$ be any sequence. Since $X \in \WASS^p_{[a,b]}\!\APD(\R,\espX)$, one can extract a subsequence $(s_n)\subset (s'_n)$ such that for any interval $[a,b]$, the sequence $\CCO{\sup_{t\in[a,b]}\norm{X(t+s'_n)}^p}_{n}$ is UI. In particular, the sequence  $\norm{X(s'_n)}^p$ is UI.  It follows that the family $(\norm{X(t)}^p)_{t\in\R}$ is UI.
\end{enumerate}
}
\end{remark}
\begin{definition}[\cite{bed-chal-mel-prf-sma2015}]{\em
Let $p\geq 0$.
We say that $X$ is
{\em $\mu$-pseudo almost periodic in  $p$-UI distribution}
if $X$ can be written
$$X=Y+Z, \text{ where }Y\in \APD^p(\R,\espX) \text{ and }
Z\in\ER(\R,\ellp{p}(\esprob,\prob,\espX),\mu).$$}
\end{definition}
The set of $\espX$-valued processes which are $\mu$-pseudo almost periodic in $p$-UI distribution is denoted by
$\PAPD^p(\R,\espX)$.

Of course, thanks to the Bochner transform, it is not difficult to extend these definitions in  Stepanov sense for a stochastic process. This will not be useful in this article, since our aim is to improve and generalize the results obtained by Da Prato and Tudor \cite{DaPrato-Tudor95}, see also \cite{KMRF12averaging} and \cite{bed-chal-mel-prf-sma2015}, and to show the absence of purely Stepanov almost periodic in distribution to stochastic differential equations. This is the purpose of Section~\ref{sec:SDE} below.


\section{Almost periodic solutions to stochastic differential
  equations with Stepanov almost periodic coefficients}
\label{sec:SDE}

Let
 $(\h_1, \|.\|_{\h_1})$ and $(\h_2, \|.\|_{\h_2})$ be separable
Hilbert spaces, and let us denote by
 $\espL(\h_1, \h_2)$ (or $\espL(\h_1)$ if
 $\h_1= \h_2$) the space of all bounded linear operators from
 $\h_1$ to $\h_2$, and by $\espLH(\h_1, \h_2)$ the space of Hilbert-Schmidt
operators from $\h_1$ to $\h_2$. For a symmetric nonnegative operator $Q\in  \espLH(\h_1, \h_2)$ with finite trace, we
assume that $W(t)$, $t\in \R$, is a $Q$-Wiener process with values on $\h_1$  defined on a stochastic basis
 $(\esprob,\tribu,(\tribu_t)_{t\in \R},\prob)$. We denote by $\trace Q$ the trace of $Q$.

Throughout this section, we consider the following abstract semilinear stochastic differential equation
\begin{equation}\label{eq:SDE}
 dX_t = AX(t)\,dt + F(t, X(t))\,dt + G(t, X(t))\,dW(t), \ t\in\R
\end{equation}
where $A: \Dom(A)\subset\h_2\rightarrow \h_2$ is a densely defined closed
(possibly unbounded) linear operator, and
$F : \R\times\h_2 \rightarrow\h_2$,
and $G : \R\times\h_2\rightarrow \espLH(\h_1, \h_2)$
are measurable functions (not necessarily continuous).
To discuss the existence and uniqueness of
($\mu$-pseudo) almost periodic in $2-$UI distribution solutions to equation \eqref{eq:SDE}, we consider the following requirements.
\begin{enumerate}
 \item[$\rm(H1)$] \label{cond:semigrp} The operator $A : \Dom(A)\rightarrow \h_2$ generates a contraction $C_0$-semigroup $(S(t))_{t\geq0}$ \cite{dapratozabczyk14book}, that is, there exists $\delta>0$ such that
$$\|S(t)\|_{\espL(\h_2)}\leq e^{-\delta t},\, \text{ for all } t\geq0.$$
  \item[$\rm(H2)$] \label{cond:croissance}The mappings $F$ and $G$ satisfy a sublinear growth condition: there exists a positive constant  $M$ such that
$$\|F(t,x)\|_{\h_{2}} + \|G(t,x)\|_{\espLH(\h_1, \h_2)} \leq M (1+\|x\|_{\h_2})$$for all $t\in \R$ and $x\in \h_2$.
\item[$\rm(H3)$] \label{cond:lipschitz} The mappings $F$ and $G$ are Lipschitz in the sense that there exists a positive function $K(.) \in \St^{p}(\R)$, with $p>2$,  such that
$$\|F(t,x) - F(t,y)\|_{\h_2} + \|G(t,x) - G(t,y)\|_{\espLH(\h_1, \h_2)}\leq
K(t)\|x - y\|_{\h_2}$$ for all $t\in \R$
and $x, y \in \h_2$.
\item[$\rm(H4)$] \label{cond:SAP}
$F \in \StAP{2}(\R\times \h_2, \h_2) $ and
$G \in \StAP{2}(\R\times \h_2, \espLH(\h_1, \h_2))$.
\end{enumerate}
\subsection{Almost periodic solutions in $2$-UI distribution}

In order to study the
$\mu$-pseudo almost periodicity of
solutions to \eqref{eq:SDE}, we need a result on almost periodicity. In what follows, let $q>0$ with $\frac{1}{2}=\frac{1}{q}+\frac{1}{p}$. Recall that $\CUB\bigl(\R,\ellp{2}(\prob, \h_2)\bigl)$, the Banach space of square-mean continuous and $\ellp{2}$-bounded
stochastic processes, is endowed with the norm
$$\|X\|_\infty^2 = \sup_t\expect\|X(t)\|_{\h_2}^2.$$

The main result concerning the existence (and uniqueness) of an almost periodic in $2$-UI distribution solution to \eqref{eq:SDE}, under Stepanov almost periodicity condition on the coefficients, is established by the next theorem.
\begin{theo}\label{theo:main}
\begin{enumerate}
  \item Under assumptions $\rm{(H1)-(H3)}$, Eq. \eqref{eq:SDE} admits a unique mild solution $X$ to \eqref{eq:SDE} in $\CUB\bigl(\R,\ellp{2}(\prob, \h_2)\bigl)$ provided that
 $$\theta_{\St} :=\CCO{\dfrac{2\norm{K}^2_{\St^2}}{\delta(1-e^{-\delta})}+\dfrac{2\norm{K}^2_{\St^2} \trace Q}{1-e^{-2\delta}}}<1.$$
Moreover, $X$ has a.e.~continuous
trajectories and can be written as,
\begin{equation}\label{eq:mildsol}
X(t) = \int^{t}_{-\infty}S(t-s)F\bigl(s, X(s)\bigl)ds +
\int^{t}_{-\infty}S(t-s)G\bigl(s, X(s)\bigl)dW(s), \,t\in \R.
\end{equation}
  \item Suppose, in addition, that $\rm{(H4)}$ is fulfilled and
$$\theta'_{\St}:= \frac{4}{3q\delta}\Bigl(\bigl(3\beta_1\bigr)^{\frac{q }{2}}+\bigl(3\beta_2\bigr)^{\frac{q }{2}}\Bigr)<1,$$ with $$\beta_1: = \frac{4}{\delta}\CCO{\frac{\norm{K}_{\St^{p}}^{p}}{1-e^{-\frac{p\delta}{4}}}}^{\frac{2}{p}},\,
\beta_2:= 4\trace Q\CCO{\frac{\norm{K}_{\St^{p}}^{p}}{1-e^{-\frac{p\delta}{2}}}}^{\frac{2}{p}}$$
then $X$ is almost
periodic in $2$-UI distribution.
\end{enumerate}
\end{theo}
Before giving the proof of Theorem \ref{theo:main}, let us state two results:
\begin{lemma}\label{lem:int-Step-vs-Bohr}
Let $K:\R\rightarrow \R$ be a nonnegative  $\St^p$-bounded (resp. Stepanov almost periodic) function, then the
function $$ \kappa(t)=\int_{-\infty}^{t}e^{-\delta(t-s)}K^p(s)ds$$ is uniformly bounded (resp. Bohr almost
periodic), and we have $$\sup_{t\in \R} \kappa(t)\leq \frac{\norm{K}_{\St^p}^{p}}{1-e^{-\delta}}.$$
\end{lemma}
\proof The proof is very simple, see for instance \cite{rao_almost_1999,rao_higher-order_2004}. 
\finpr
The following proposition is based on the application of Koml\'os's theorem \cite{komlos67gene}.
\begin{prop}\label{lem:limits_inversion}
1.
Let $(\espmes,\trmes,\lambda)$ be a $\sigma$-finite measure space,
and let $\espB$ be a
separable Banach space.
Let $(\ff_n)$ be a sequence of mappings from $\espmes\times\espB$ to
$\espB$ satisfying
\begin{enumerate}[(i)]
\item each $\ff_n$ is measurable, 

\item for every $\elmes\in\espmes$, the sequence $(\ff_n(\elmes,.))$ is
  equicontinuous,

\item there exists a measurable mapping $\fg :\,\espmes\times\espB\rightarrow
  \espB $, 
such that
\begin{equation*}
(\forall x\in\espB)\
  \lim_{n\rightarrow\infty}
    \int_\espmes \norm{\ff_n(\elmes,x)-\fg(\elmes,x)}\,d\lambda(\elmes)=0.
\end{equation*}
\setcounter{deplus}{\value{enumi}}
\end{enumerate}
Then there exists a subsequence $(\fg_{n})$ of $(\ff_n)$,
a modification $\fgg$ of $\fg$,
and a $\lambda$-negligible subset $\mnegl$ of $\espmes$ such that
\begin{equation*}
(\forall x\in\espB)\
(\forall \elmes\in\espmes\setminus\mnegl)\
  \lim_{n\rightarrow\infty}
    \fg_n(\elmes,x)=\fgg(\elmes,x),
\end{equation*}
and such that, for every $\elmes\in\espmes\setminus\mnegl$,
the mapping $\fg(\elmes,.)$ is
  continuous.

2. With the same hypothesis as in 1.,
assume now that $\espmes=\R$ is the set of real numbers,
$\trmes$ its Borel $\sigma$-algebra, and $\lambda$ the
Lebesgue-measure.
Assume furthermore that there exists a sequence $(\fK_n)$ of measurable mappings
from $\R$ to $\R^{+}$, and a number $p\geq 1$,
satisfying
\begin{enumerate}[(i)]
\setcounter{enumi}{\value{deplus}}
\item
\begin{equation*}
(\forall n\geq 1)\
(\forall x,y\in\espB)\
(\forall \elmes\in\R)\
\norm{\ff_n(\elmes,x)-\ff_n(\elmes,y)}\leq \fK_n(\elmes)\norm{x-y},
\end{equation*}
\item
\begin{equation*}
\Mmes:=\sup_{n\geq 1}\sup_{\elmes\in\R} \int_\elmes^{\elmes+1}{(\fK_{n})}^{p}(\elmess)\,d\elmess<\infty.
\end{equation*}
\end{enumerate}
Then we can extract the subsequence $(\fg_{n})$ in such a way that
there exists a measurable mapping $\fKg
:\,\R\rightarrow\R^{+}$
and a
$\lambda$-negligible subset $\mnegl$ of $\R$
such that
\begin{equation}\label{eq:limlip}
(\forall x,y\in\espB)\
(\forall \elmes\in\espmes\setminus\mnegl)\
\norm{\fg(\elmes,x)-\fg(\elmes,y)}\leq \fKg(\elmes)\norm{x-y},
\end{equation}
with
\begin{equation*}
\sup_{\elmes\in\R} \int_\elmes^{\elmes+1}\fKg^p(\elmess)\,d\elmess\leq \Mmes.
\end{equation*}
\end{prop}
\proof

1. For every $x\in\espB$, we can find a subsequence $(\fg^{(x)}_n)$ of
$(\ff_n)$ and a $\lambda$-negligible set $\mnegl_x$ such that
\begin{equation*}
(\forall \elmes\in\espmes\setminus\mnegl_x)\
  \lim_{n\rightarrow\infty}
    \fg^{(x)}_n(\elmes,x)=\fg(\elmes,x).
\end{equation*}
Let $\mdense$ be a dense countable subset of $\espB$.
Using a diagonal procedure,
we can find a subsequence $(\fg_{n})$ of $(\ff_n)$ and a
$\lambda$-negligible subset $\mnegl$ of $\espmes$ such that
\begin{equation*}
(\forall y\in\mdense)\
(\forall \elmes\in\espmes\setminus\mnegl)\
  \lim_{n\rightarrow\infty}
    \fg_n(\elmes,y)=\fg(\elmes,y).
\end{equation*}
On the other hand,
for every $\elmes\in\espmes$ and every $x\in\espB$, we have,
by equicontinuity of $\fg_n(\elmes,.)$,
\begin{equation}\label{eq:equicont}
  \lim_{y\rightarrow x}\sup_{n}
      \norm{\fg_n(\elmes,y)-\fg_n(\elmes,x)}=0.
\end{equation}
Let $\elmes\in\espmes\setminus\mnegl$, and let $x\in\espB$.
Using the uniformity in \eqref{eq:equicont},
we deduce, by a classical result on interchange of limits, that, for any
$x\in\espB$,
\begin{equation}\label{eq:interversion}
\lim_{n\rightarrow\infty}
    \fg_n(\elmes,x)
=\lim_{n\rightarrow\infty}
  \lim_{\substack{y\rightarrow x\\y\in\mdense}}
      \fg_n(\elmes,y)
=  \lim_{\substack{y\rightarrow x\\y\in\mdense}}
 \lim_{n\rightarrow\infty}
      \fg_n(\elmes,y)
= \lim_{\substack{y\rightarrow x\\y\in\mdense}}\fg(\elmes,y).
\end{equation}
Note that, for $\elmes\in\espmes\setminus\mnegl$, the calculation
\eqref{eq:interversion} shows that $f(u,.)$
is continuous on $\mdense$.
Let us define $\fgg :\,\espmes\times\espB\rightarrow\espB$ by
$$\fgg(\elmes,x)=\left\{\begin{array}{ll}
\displaystyle
\lim_{n\rightarrow\infty}
    \fg_n(\elmes,x)
=\lim_{\substack{y\rightarrow
    x\\y\in\mdense}}\fg(\elmes,y)
 &\text{ for }
\elmes\in\espmes\setminus\mnegl\text{ and }x\in\espB,\\
0&\text{ for }
\elmes\in\mnegl\text{ and }x\in\espB.
\end{array}\right.$$
This definition is consistent, thanks to
\eqref{eq:interversion}. Furthermore, $\fgg(\elmes,.)$ is continuous
on $\espB$ for every $\elmes\in\espmes$.
Finally,
since $\fgg(\elmes,y)=\fg(\elmes,y)$
for all $(\elmes,y)\in(\espmes\setminus\mnegl)\times\mdense$,
we have, for any $x\in\espB$,
\begin{multline*}
\int_\espmes \norm{\fgg(\elmes,x)-\fg(\elmes,x)}\,d\lambda(\elmes)\\
\leq \lim_{\substack{y\rightarrow
    x\\y\in\mdense}}
\big\lgroup
   \int_\espmes \norm{\fg_n(\elmes,y)-\fgg(\elmes,x)}\,d\lambda(\elmes)
+
   \int_\espmes \norm{\fg_n(\elmes,y)-\fg(\elmes,x)}\,d\lambda(\elmes)
\big\rgroup=0,
\end{multline*}
which proves that $\fgg(\elmes,x)=\fg(\elmes,x)$ for $\lambda$-almost
every $\elmes\in\espmes$.

\medskip
2. By an application of Koml\'os's theorem \cite{komlos67gene} on each
interval $[k,k+1]$, where $k$ is an integer,
and using a diagonal procedure, we can extract
a subsequence $(\fKg_n)$ of
$(\fK_n)$ and a mapping $\fKg :\,\R\rightarrow\R^{+}$
such that
\begin{equation*}
  \lim_{n\rightarrow\infty}
          \frac{1}{n}\sum_{j=1}^{n}\fKg_j^p(\elmes)=\fKg^p(\elmes)\quad
          \text{for $\lambda$-a.e. }\elmes\in\R,
\end{equation*}
and such that this almost sure Ces\`aro convergence holds true for any
further
subsequence of $(\fKg_n)$
(the negligible set on which the convergence does not hold depends on
the subsequence).
We can thus ask for the sequences $(\fg_n)$ and $(\fKg_n)$ to have
the same indices.
Denote, for $n\geq 1$,
\begin{equation*}
\fgC_n=\frac{1}{n}\sum_{j=1}^{n}\fg_j,\quad
\fKgC_n=\CCO{\frac{1}{n}\sum_{j=1}^{n}\fKg_j^p}^{1/p}.
\end{equation*}
There exists a $\lambda$-negligible subset $\mnegl$ of $\R$ such that
\begin{equation}\label{eq:1lim}
(\forall x\in\espB)\
(\forall \elmes\in\R\setminus\mnegl)\
  \lim_{n\rightarrow\infty}
    \fgC_n(\elmes,x)=\fg(\elmes,x) \text{ and }
  \lim_{n\rightarrow\infty}
    \fKgC_n(\elmes)=\fKg(\elmes).
\end{equation}
On the other hand, by the triangle inequality, we have also:
\begin{equation}\label{eq:2lip}
(\forall n\geq 1)\
(\forall x,y\in\espB)\
(\forall \elmes\in\R)\
\norm{\fgC_n(\elmes,x)-\fgC_n(\elmes,y)}\leq \fKgC_n(\elmes)\norm{x-y}.
\end{equation}
We deduce \eqref{eq:limlip} from \eqref{eq:1lim} and \eqref{eq:2lip}.
Furthermore, by Fatou's lemma, we have
\begin{equation*}
\sup_{\elmes\in\R} \int_\elmes^{\elmes+1}\fKg^p(\elmess)\,d\elmess
\leq \sup_{\elmes\in\R} \liminf_{n\rightarrow\infty}\int_\elmes^{\elmes+1}\fKgC_n^p(\elmess)\,d\elmess
\leq\Mmes.
\end{equation*}

\finpr

 \proofof{Theorem \ref{theo:main}.} Clearly, the process
$$X(t) = \int^{t}_{-\infty}T(t-s)F\bigl(s, X(s)\bigl)ds +
\int^{t}_{-\infty}T(t-s)G\bigl(s, X(s)\bigl)dW(s)$$ satisfies
 $$X(t) = T(t-a)X(a)+ \int^{t}_{a}T(t-s)F\bigl(s, X(s)\bigl)ds
+ \int^{t}_{a}T(t-s)G\bigl(s, X(s)\bigl)dW(s)$$ for all $t\geq a$ for each $a\in \R$ , and hence $X$ is a mild
solution to \eqref{eq:SDE}.
We introduce an operator $\Gamma$ by
$$\Gamma X(t) = \int^{t}_{-\infty}T(t-s)F\bigl(s, X(s)\bigl)ds +
\int^{t}_{-\infty}T(t-s)G\bigl(s, X(s)\bigl)dW(s).$$
The rest of the proof is shared naturally into three steps.

\medskip
\noindent{\em \textbf{Step~1: existence and uniqueness of a mild solution in $\CUB\bigl(\R,\ellp{2}(\prob, \h_2)\bigl)$.}} Let us show that $\Gamma$ has a unique fixed point. For this purpose we need to show that $\Gamma$
maps $\St^2\bigl(\R,\ellp{2}(\prob,
\h_2)\bigl)$ into $\CUB\bigl(\R,\ellp{2}(\prob, \h_2)\bigl)$. Let $X \in \St^2\bigl(\R,\ellp{2}(\prob,
\h_2)\bigl)$.
Put $\Gamma=\Gamma_1+\Gamma_2$, where
$$(\Gamma_1X)(t)=\int^{t}_{-\infty}T(t-s)F\bigl(s, X(s)\bigl)ds$$ and $$(\Gamma_2X)(t)=\int^{t}_{-\infty}T(t-s)G\bigl(s,
X(s)\bigl)dW(s).$$ Using conditions (H1) and (H3),  the functions $F$ and $G$ satisfy the properties $f(.):=F(.,X(.)) \in
\St^2\bigl(\R,\ellp{2}(\prob, \h_2)\bigl)$ and $g(.):=G(.,X(.)) \in \St^2\bigl(\R,\ellp{2}(\prob,\espLH(\h_1,\h_2)\bigl)$.
Let us introduce the following processes, for each $n\geq 1$,
$$(\Gamma_{1,n}X)(t)=\int_{t-n}^{t-n+1}T(t-s)f(s)ds$$ and $$(\Gamma_{2,n} X)(t)=\int_{t-n}^{t-n+1}T(t-s)g(s)dW(s).$$
Clearly, for each $n$, $\Gamma_{1,n}X \in \Cont(\R,\ellp{2}(\prob, \h_2))$. Likewise for $\Gamma_{2,n}X$, for which the continuity is a property of the
stochastic integral. To show the boundedness of $\Gamma_{1,n}X $ and $\Gamma_{2,n}X $ for each fixed $n\geq 1$, we use standard arguments.
By
Hölder's inequality, we have, for any $t\in\R$, and $n \geq 1$
\begin{align*}
\expect\norm{(\Gamma_{1,n} X)(t)}_{\h_2}^2&\leq \expect\CCO{ \int^{t-n+1}_{t-n}\norm{T(t-s)}\norm{f(s)}_{\h_2} ds}^2\\
&\leq \delta^{-1}\int^{t-n+1}_{t-n}e^{-\delta (t-s)}\expect \norm{f(s)}^2_{\h_2} ds\\
&\leq \delta^{-1} e^{-\delta (n-1)}\int^{t-n+1}_{t-n}\expect \norm{f(s)}^2_{\h_2} ds,
\end{align*}
 which leads to
\begin{equation*}
\norm{\Gamma_{1,n} X}_{\infty}^2\leq \delta^{-1}e^{-\delta (n-1)}\norm{f}_{\St^2}^2.
\end{equation*}
Since the series $\sum_{n=1}^\infty e^{-2\delta (n-1)}\norm{f}_{\St^2}^2$ is
convergent, it follows that
\begin{equation}\label{eq:CUB}
 \Gamma_1X:=\sum_{n=1}^\infty \Gamma_{1,n} X \in \CUB\bigl(\R,\ellp{2}(\prob, \h_2)\bigl).
\end{equation}
By Itô's isometry, we have for every $t\in\R$ and $n \geq 1$,
\begin{multline*}
\begin{aligned}
\expect\norm{(\Gamma_{2,n} X)(t)}_{\h_2}^2&= \trace Q \int^{t-n+1}_{t-n}\expect\norm{T(t-s)}^2\norm{g(s)}_{\espLH(\h_1, \h_2)} ^2ds\\
& \leq \trace Q \int^{t-n+1}_{t-n}e^{-2\delta(t-s)}\expect \norm{g(s)}^2_{\espLH(\h_1, \h_2)} ds\\
&\leq \trace Qe^{-2\delta (n-1)}\norm{g}_{\St^2}^2.
\end{aligned}
\end{multline*}
This shows that $\Gamma_{2,n} X \in \CUB\bigl(\R,\ellp{2}(\prob, \h_2)\bigl)$ for each $n\geq 1$. Since
$\sum_{n=1}^\infty e^{-2\delta (n-1)}<+\infty$, the series $\sum_{n=1}^\infty (\Gamma_{2,n}
X)(t)$ is uniformly convergent on $\R$. Thus
\begin{equation}\label{eq:CUBW}
 \Gamma_2X:=\sum_{n=1}^\infty \Gamma_{2,n} X \in
\CUB\bigl(\R,\ellp{2}(\prob, \h_2)\bigl).
\end{equation}
From (\ref{eq:CUB}) and (\ref{eq:CUBW}), we deduce that $\Gamma$
maps $\St^2\bigl(\R,\ellp{2}(\prob,\h_2)\bigl)$ into $\CUB\bigl(\R,\ellp{2}(\prob, \h_2)\bigl)$.

Let us show that $\Gamma$ is a contraction operator.
We have, for any $t\in\R$ and $X,Y \in \CUB\bigl(\R,\ellp{2}(\prob, \h_2)\bigl)$,
\begin{align*}
 \expect\norm{ (\Gamma X)(t) -  (\Gamma Y)(t)}_{\h_2}^2  \leq &2 \expect\CCO{\int_{-\infty}^{t}e^{-\delta(t-s)}\|F(s, X(s))- F(s,Y(s))\|_{\h_2}ds}^2 \\
  & + 2
\expect\norm{\int_{-\infty}^{t}T(t-s)[G(s, X(s)) - G(s,Y(s))]dW(s)}_{\h_2}^2\\
  = &  I_1(t) + I_2(t).
\end{align*}
Let us estimate $I_1(t)$. Using $\rm(H3)$, Cauchy-Schwartz inequality, and Lemma~\ref{lem:int-Step-vs-Bohr}, we obtain:
\begin{align*}
I_1(t)  & \leq 2\CCO{\int_{-\infty}^{t}e^{-\delta(t-s)}ds}\CCO{\int_{-\infty}^{t}e^{-\delta(t-s)}\expect\|F(s, X(s))-F(s,Y(s))\|_{\h_2}^2ds}\\
&\leq \frac{2}{ \delta}\CCO{\int_{-\infty}^{t}e^{-\delta (t-s)}K^2(s)\expect\|X(s)- Y(s)\|_{\h_2}^2ds}\\
&\leq\frac{2}{ \delta}\CCO{\sup_{s\in\R} \expect\|X(s)- Y(s)\|_{\h_2}^2}\CCO{\int_{-\infty}^{t}e^{-\delta(t-s)}K^2(s)ds}\\
&\leq\dfrac{2\norm{K}^2_{\St^2}}{\delta(1-\exp(-\delta))}\sup_{s\in\R} \expect\|X(s))- Y(s))\|_{\h_2}^2.
\end{align*}
For $I_2(t)$, using again $\rm(H3)$, Lemma~\ref{lem:int-Step-vs-Bohr}, and Itô's isometry we get:
\begin{align*}
I_2(t) &\leq  2\trace Q
   \int_{-\infty}^{t}e^{-2\delta(t-s)}
         \expect\|G(s, X(s)) - G(s,Y(s))\|^{2}_{\espLH(\h_1, \h_2)}ds\\
&\leq  2\trace Q
    \int_{-\infty}^{t}e^{-2\delta(t-s)}
      K^{2}(s)\expect\|X(s) - Y(s)\|^{2}_{\h_2}ds\\
&\leq 2 \trace Q\CCO{\sup_{s\in\R} \expect\|X(s)- Y(s)\|_{\h_2}^2}\CCO{\int_{-\infty}^{t}e^{-2\delta(t-s)}K^2(s)ds}\\
&\leq\dfrac{2\norm{K}^2_{\St^2} \trace Q}{1-\exp(-2\delta)}\CCO{\sup_{s\in\R} \expect\|X(s)- Y(s)\|_{\h_2}^2}.
\end{align*}
We thus have
$$\|\Gamma X -  \Gamma Y\|_{\infty}^2 \leq \CCO{\dfrac{2\norm{K}^2_{\St^2}}{\delta(1-\exp(-\delta))}+\dfrac{2\norm{K}^2_{\St^2} \trace Q}{1-\exp(-2\delta)}} \|X - Y\|_{\infty}^2=\theta _{\St}\|X - Y\|_{\infty}^2.$$
Consequently, as $\theta _{\St}< 1$, we deduce that $\Gamma$ is a contraction operator, hence there exists a unique mild
solution to \eqref{eq:SDE} in  $\CUB\bigl(\R,\ellp{2}(\prob, \h_2)\bigl)$. By \cite[Theorem 7.2]{dapratozabczyk14book}), almost all trajectories of this solution are continuous.

\medskip
\noindent{\em \textbf{Step~2: almost periodicity in one-dimensional distribution of the solution.}}
 We use Bochner's double sequences criterion in Stepanov sense.
Since $F \in \StAP{2}(\R\times \h_2, \h_2) $ and
$G \in \StAP{2}(\R\times \h_2, \espLH(\h_1, \h_2))$, we deduce, by Proposition~\ref{prop:stepanov-ponctuel}, that there exist
subsequences  $(\alpha_n^{'}) \subset (\alpha^{''}_n)$ and $(\beta_n^{'}) \subset (\beta^{''}_n)$ with same indexes (and independent of $x$), and
functions $F^{\infty}\in \StAP{2}(\R\times \h_2, \h_2)$ and  $G^{\infty} \in \StAP{2}(\R\times \h_2, \espLH(\h_1, \h_2))$ such that for every $t\in \R$ and $x\in \h_2$
\begin{multline}\label{eq:SAPF}
 \lim_{n\rightarrow\infty}\int_t^{t+1}\norm{F(s+\alpha_n^{'}+\beta_n^{'},
 x)-F^{\infty}(s,x)}_{\h_2}^2ds \\
 =\lim_{n\rightarrow\infty}\lim_{m\rightarrow\infty}\int_t^{t+1}\norm{F(s+\alpha_n^{'}+\beta_m^{'},
 x)-F^{\infty}(s,x)}_{\h_2}^2ds =0,
\end{multline}
\begin{multline}\label{eq:SAPG}
 \lim_{n\rightarrow\infty}\int_t^{t+1}\norm{G(s+\alpha_n^{'}+\beta_n^{'},
 x)-G^{\infty}(s,x)}_{\espLH(\h_1,\h_2)}^2ds \\
= \lim_{n\rightarrow\infty}\lim_{m\rightarrow\infty}\int_t^{t+1}\norm{G(s+\alpha_n^{'}+\beta_n^{'},
 x)-G^{\infty}(s,x)}_{\espLH(\h_1,\h_2)}^2ds =0.
\end{multline}
These limits exist also uniformly with respect to $t \in \R$.

Thanks to Proposition~\ref{lem:limits_inversion}, we obtain the following interesting properties:
\begin{itemize}
 \item The functions $F^{\infty}$ and $G^{\infty}$ satisfy similar conditions as (H2) and (H3).
 \item There are subsequences of  $(\alpha_n)$ (resp. $(\beta_n))$, still noted (for simplicity) by $(\alpha_n)$ (resp. $(\beta_n))$ and
Lebesgue-negligible subset $\mnegl$ of $\R$ such that for all $s \in \R \setminus \mnegl$ and every $ x \in \h_2$
\begin{eqnarray}\label{eq:APF}
 \lim_{n\rightarrow\infty}F(s+\alpha_n+\beta_n,
 x) = \lim_{n\rightarrow\infty}\lim_{m\rightarrow\infty}F(s+\alpha_n+\beta_m,
 x) = F^{\infty}(s,x),\\
 \lim_{n\rightarrow\infty}G(s+\alpha_n+\beta_n,
 x) = \lim_{n\rightarrow\infty}\lim_{m\rightarrow\infty}G(s+\alpha_n+\beta_m,
 x) = G^{\infty}(s,x).\label{eq:APG}
\end{eqnarray}
\end{itemize}
We now set $\gamma_n = \alpha_n+\beta_n$ and
consider the sequence of operators, defined, for each $n\geq 1$, by
$$(\Gamma^nX)(t) = \int^t_{-\infty}T(t-s)F(s+\gamma_n, X(s))ds
     + \int^t_{-\infty}T(t-s)G(s+\gamma_n, X(s))dW(s).$$ Let $\Gamma^\infty$ be the operator defined by $$(\Gamma^\infty X)(t) = \int^t_{-\infty}T(t-s)F^\infty(s, X(s))ds
+ \int^t_{-\infty}T(t-s)G^\infty(s, X(s))dW(s).$$
Using the same reasoning as in the first step, we deduce that, for each $n\geq 1$, $\Gamma^n$ maps $\St^2\bigl(\R,\ellp{2}(\prob,
\h_2)\bigl)$ into $\CUB\bigl(\R,\ellp{2}(\prob, \h_2)\bigl)$ and it is a contraction operator, with  contraction constant equal to $\theta _{\St}$.
It follows, by an application of the fixed point theorem,  that there exists a process
$$X^n(t) = \int^t_{-\infty}T(t-s)F(s+\gamma_n, X^n(s))ds
     + \int^t_{-\infty}T(t-s)G(s+\gamma_n, X^n(s))dW(s)$$
which is the fixed point of $\Gamma^n$ and also the mild solution to
\begin{equation*}
 dX(t) = A X(t)dt + F(t+\gamma_n, X(t))dt + G(t+\gamma_n, X(t))dW(t).
\end{equation*}
Moreover, thanks to Proposition~\ref{lem:limits_inversion}, the mappings $F^\infty$ and $G^\infty$ satisfy similar conditions as (H2) and (H3). Hence, the fixed point theorem applied on  $\Gamma^\infty$ ensures the existence of a process $X^\infty$, satisfying the integral equation $$X^\infty(t) = \int^t_{-\infty}T(t-s)F^\infty(s, X^\infty(s))ds + \int^t_{-\infty}T(t-s)G^\infty(s, X^\infty(s))dW(s),$$
that is, $X^\infty$ is a mild solution to
\begin{equation*}
 dX(t) = A X(t)dt + F^\infty(t, X(t))dt + G^\infty(t, X(t))dW(t).
\end{equation*}

Make the change of variable $\sigma + \gamma_n = s$, the process
\begin{equation*}
X(t+ \gamma_n)
= \int^{t+\gamma_n}_{-\infty}T(t+\gamma_n-\sigma)F(\sigma, X(\sigma))d\sigma
     + \int^{t+\gamma_n}_{-\infty}T(t+\gamma_n-\sigma)G(\sigma, X(\sigma))dW(\sigma)
\end{equation*}
becomes
\begin{equation*}
X(t+ \gamma_n)= \int^t_{-\infty}T(t-s)F(s+\gamma_n, X(s+\gamma_n))ds
     + \int^t_{-\infty}T(t-s)G(s+\gamma_n,X(s+\gamma_n))d\tilde{W}_n(s),
\end{equation*}
where  $\tilde{W}_n(s) = W(s+\gamma_n) - W(\gamma_n)$ is a Brownian motion with the same distribution as $W(s)$.
We deduce that the process $X(t+\gamma_n)$ has the same
distribution as $X^n(t)$.

Let us show that $X^n(t)$ converges in quadratic mean to $X^\infty(t)$ for each
fixed $t\in \R$.  We have, by the triangular inequality,
\begin{align*}
 \expect\lVert X^n(t) - X^\infty(t)\rVert^2 =& \expect\lVert \int^t_{-\infty}T(t-s)[F(s+\gamma_n, X^n(s)) -F^\infty(s, X^\infty(s))]ds\\
&+\int^t_{-\infty}T(t-s)[G(s+\gamma_n, X^n(s)) - G^\infty(s, X^\infty(s))]dW(s)\rVert^2\\
\leq&4 \expect\lVert \int^t_{-\infty}T(t-s)[F(s+\gamma_n, X^n(s)) -F(s+\gamma_n, X^\infty(s))]ds\rVert^2\\
&+4 \expect\lVert\int^t_{-\infty}T(t-s)[G(s+\gamma_n, X^n(s)) - G(s+\gamma_n, X^\infty(s))]dW(s)\rVert^2\\
&+4\expect\lVert \int^t_{-\infty}T(t-s)[F(s+\gamma_n, X^\infty(s)) -F^\infty(s, X^\infty(s))]ds\rVert^2\\
&+4\expect\lVert\int^t_{-\infty}T(t-s)[G(s+\gamma_n, X^\infty(s)) - G^\infty(s, X^\infty(s))]dW(s)\rVert^2 \\
\leq& I_1^n(t) + I_2^n(t) + I_3^n(t)+ I_4^n(t).
\end{align*}

Now, using (H1), (H3), Hölder's inequality, and Lemma~\ref{lem:int-Step-vs-Bohr}, we obtain
\begin{align*}
I_1^n(t)&= 4\expect\lVert \int^t_{-\infty}T(t-s)
      [F(s+\gamma_n, X^n(s)) -F(s+\gamma_n, X^\infty(s))]ds\rVert^2\\
&\leq4\expect\CCO{\int^t_{-\infty}
  e^{-\delta(t-s)} \lVert F(s+\gamma_n, X^n(s)) -F(s+\gamma_n, X^\infty(s))
  \rVert ds}^2\\
&\leq\frac{4}{\delta}\int^t_{-\infty}e^{-\delta(t-s)}K^2(s+\gamma_n)\expect\lVert  X^n(s)
               - X^\infty(s)\rVert^2 ds\\
&\leq \frac{4}{\delta}\Bigl(\int^t_{-\infty}e^{-\frac{p\delta}{4}(t-s)}K^{p}(s+\gamma_n)ds\Bigr)^{\frac{2}{p}}
\Bigl(\int^t_{-\infty}e^{-\frac{q\delta}{4}(t-s)}
\bigl(\expect\lVert  X^n(s) - X^\infty(s)\rVert^2\bigr)^{\frac{q}{2}} ds\Bigr)^{\frac{2}{q}}\\
&\leq \frac{4}{\delta}\Biggl(\frac{\norm{K}_{\St^{p}}^{p}}{1-e^{-\frac{p\delta}{2}}}\Biggr)^{\frac{2}{p}}
\Biggl(\int^t_{-\infty}e^{-\frac{q\delta}{2}(t-s)}
\bigl(\expect\lVert  X^n(s) - X^\infty(s)\rVert^2\bigr)^{\frac{q}{2}} ds\Biggr)^{\frac{2}{q}}.
\end{align*}
For $I_2^n(t)$, using Itô's isometry,  Hölder's inequality, and Lemma~\ref{lem:int-Step-vs-Bohr}, we get
\begin{align*}
 I_2^n(t)& = 4\expect\lVert\int^t_{-\infty}T(t-s)
    [G(s+\gamma_n, X^n(s)) -G(s+\gamma_n, X^\infty(s))]dW(s)\rVert^2\\
& \leq 4\trace Q\expect\int^t_{-\infty}\lVert T(t-s)\rVert^2
      \lVert G(s+\gamma_n, X^n(s)) -G(s+\gamma_n, X^\infty(s))\rVert^2 ds\\
&\leq 4\trace Q\int^t_{-\infty}e^{-2\delta(t-s)}K^2(s+\gamma_n)
      \expect\lVert X^n(s) - X^\infty(s)\rVert^2ds\\
  &\leq 4\trace Q \CCO{\int^t_{-\infty}e^{-\frac{p\delta}{2}(t-s)}K^{p}(s+\gamma_n)ds}^{\frac{2}{p}}
\CCO{\int^t_{-\infty}e^{-\frac{q\delta }{2}(t-s)}
\bigl(\expect\lVert  X^n(s) - X^\infty(s)\rVert^2\bigr)^{\frac{q}{2}} ds}^{\frac{2}{q}}\\
&\leq 4\trace Q\Biggl(\frac{\norm{K}_{\St^{p}}^{p}}{1-e^{-\frac{p\delta}{2}}}\Biggr)^{\frac{2}{p}}
\Biggl(\int^t_{-\infty}e^{-\frac{q\delta}{2}(t-s)}
\bigl(\expect\lVert  X^n(s) - X^\infty(s)\rVert^2\bigr)^{\frac{q}{2}} ds\Biggr)^{\frac{2}{q}}.
\end{align*}
Let us show that $I_3^n(t)$ and $I_4^n(t)$ go to $0$ as $n$ goes to infinity.

For any $r\in\R$, since $X^\infty\in\CUB\Bigl(\R,\ellp{2}(\prob, \h_2)\Bigr)$,
the family
$$\Bigl(\norm{X^\infty(s)}^2\Bigr)_{r\leq s\leq r+1}$$
is uniformly integrable, by the converse to Vitali's theorem.
By the growth condition satisfied by $F$ and $F^\infty$, this shows that the
family
$$(U_{s,n}):=\Bigl(
         \lVert F(s+\gamma_n, X^\infty(s)) -F^\infty(s, X^\infty(s))\rVert^2
                                                  \Bigr)_{r\leq s\leq
                                                    r+1,\, n\geq 1}$$
is uniformly integrable.
By La Vall\'ee Poussin's criterion, there
exists a non-negative increasing convex function $\Phi :
\R\rightarrow\R$
such that
$\lim_{t\rightarrow\infty}\frac{\Phi(t)}{t}=+\infty$ and
$\sup_{s,n}\expect(\Phi(U_{s,n}))<+\infty$.
We thus have
$$\sup_{n}\expect \int^{r+1}_{r}
     \Phi\bigl(U_{s,n}\bigr)\,ds <+\infty,$$
which prove that the family
$(U_{.,n})_{n\geq 1}$
is uniformly integrable with respect to the probability measure
$\prob\otimes \lambda$ on
$\esprob\times[r,r+1]$, where $\lambda$ denotes Lebesgue's measure.
This proves that, for any $r\in\R$,
\begin{equation}\label{eq:UIconv}
\lim_{n\rightarrow+\infty}\Bigg\lgroup\expect\CCO{ \int^{r+1}_{r}
     \lVert F(s+\gamma_n, X^\infty(s))- F^\infty(s, X^\infty(s))\rVert^2
      ds}      \Bigg\rgroup^{1/2} =0.
\end{equation}

Let $t\geq 0$. Since $X^\infty\in\CUB\bigl(\R,\ellp{2}(\prob,
\h_2)\bigl)$, and thanks to the growth condition satisfied by $F$, the
sequence
$$\Bigl(\expect\int^{t-k+1}_{t-k}\lVert
  F(s+\gamma_n, X^\infty(s))- F^\infty(s, X^\infty(s))\rVert^{2}
  ds\Bigr)_{k\geq 1, n\geq 0}$$
is bounded. We can thus find an integer $N(t,\eta)$ such that, for any
$n\geq 0$,
\begin{equation}\label{eq:Nteta-}
\bigg\lgroup\sum_{k>N(t,\eta)}e^{-\delta(k-1)}\expect \int^{t-k+1}_{t-k}\lVert
  F(s+\gamma_n, X^\infty(s))- F^\infty(s, X^\infty(s))\rVert^{2} ds
\bigg\rgroup^{1/2}
\leq\eta.
\end{equation}
Using \eqref{eq:Nteta-}, we get
\begin{align}
  \sqrt{ I^n_3 (t)}
\leq&2\bigg\lgroup\expect\CCO{\int^t_{-\infty}\lVert T(t-s)\rVert
     \lVert F(s+\gamma_n, X^\infty(s))- F^\infty(s, X^\infty(s))\rVert
      ds}^2
     \bigg\rgroup^{1/2} \notag\\
\leq&2\sum^{N(t,\eta)}_{k=1}
     \Bigg\lgroup\expect\CCO{ \int^{t-k+1}_{t-k}e^{-\delta(t-s)}
     \lVert F(s+\gamma_n, X^\infty(s))- F^\infty(s, X^\infty(s))\rVert
      ds}^2      \Bigg\rgroup^{1/2} \notag\\
&+2\sum_{N(t,\eta)+1}^{\infty}\Bigg\lgroup\expect\CCO{ \int^{t-k+1}_{t-k}e^{-\delta(t-s)}
     \lVert F(s+\gamma_n, X^\infty(s))- F^\infty(s, X^\infty(s))\rVert
      ds}^2      \Bigg\rgroup^{1/2} \notag\\
\leq&2\sum^{N(t,\eta)}_{k=1} e^{-\delta(k-1)}
     \Bigg\lgroup\expect\CCO{ \int^{t-k+1}_{t-k}
     \lVert F(s+\gamma_n, X^\infty(s))- F^\infty(s, X^\infty(s))\rVert
      ds}^2      \Bigg\rgroup^{1/2} +2\eta.\label{eq:In}
\end{align}
Since the sum in \eqref{eq:In} is finite and $\eta$ is arbitrary,
we deduce from \eqref{eq:UIconv} that
$$\lim_{n\rightarrow+\infty}I^n_3 (t)=0.$$

For $I_4^n(t)$, applying Itô's isometry, we obtain
 \begin{align*}
 I_4^n(t)& = 4\expect\lVert\int^t_{-\infty}T(t-s)
     [G(s+\gamma_n, X^\infty(s))-G^\infty(s, X^\infty(s))]dW(s)\rVert^2\\
&\leq 4\trace Q\expect\CCO{\int^t_{-\infty}e^{-2\delta(t-s)}
      \lVert G(s+\gamma_n, X^\infty(s))-G^\infty(s, X^\infty(s))\rVert^2ds}.
\end{align*}
For the same reason as for $I_3^n(t)$ and by \eqref{eq:APG}, $I_4^n(t)$ goes to $0$ as $n\rightarrow \infty$.
Now, let us define the following quantities:
\begin{align*}
\alpha_n(t):= I_3^n(t) + I_4^n(t),\,\,
\beta_1:=  \frac{4}{\delta}\CCO{\frac{\norm{K}_{\St^{p}}^{p}}{1-e^{-\frac{p\delta}{4}}}}^{\frac{2}{p}},\,\,
\beta_2:= 4\trace Q\CCO{\frac{\norm{K}_{\St^{p}}^{p}}{1-e^{-\frac{p\delta}{2}}}}^{\frac{2}{p}}.
\end{align*}
From the above, we have
\begin{multline*}
\expect \lVert X^n(t) - X^\infty(t)\rVert^2 \leq \alpha_n(t) + \beta_1\Biggl(\int^t_{-\infty}e^{-\frac{q\delta}{4}(t-s)}
\bigl(\expect\lVert  X^n(s) - X^\infty(s)\rVert^2\bigr)^{\frac{q}{2}} ds\Biggr)^{\frac{2}{q}}\\
\begin{aligned}
 +
\beta_2\Biggl(\int^t_{-\infty}e^{-\frac{q\delta}{2}(t-s)}
\bigl(\expect\lVert  X^n(s) - X^\infty(s)\rVert^2\bigr)^{\frac{q}{2}} ds\Biggr)^{\frac{2}{q}}.
\end{aligned}
\end{multline*}
By convexity of the mapping $u\mapsto u^{\frac{q}{2}}$ defined on $\R^+$, we get
\begin{multline*}
 \bigl(\expect\lVert X^n(t) - X^\infty(t)\rVert^2\bigr)^{\frac{q}{2}}
 \leq \frac{1}{3}\bigl(3\alpha_n(t)\bigr)^{\frac{q}{2}} + \frac{1}{3}\bigl(3\beta_1\bigr)^{\frac{q}{2}}
 \Biggl(\int^t_{-\infty}e^{-\frac{q\delta}{4}(t-s)}
\bigl(\expect\lVert  X^n(s) - X^\infty(s)\rVert^2\bigr)^{\frac{q}{2}} ds\Biggr)\\
\begin{aligned}
+
\frac{1}{3}\bigl(3\beta_2\bigr)^{\frac{q}{2}}\Biggl(\int^t_{-\infty}e^{-\frac{q\delta}{2}(t-s)}
\bigl(\expect\lVert  X^n(s) - X^\infty(s)\rVert^2\bigr)^{\frac{q}{2}} ds\Biggr).
\end{aligned}
\end{multline*}
Since $\theta'_{\St}= \frac{4}{3q\delta}\Bigl(\bigl(3\beta_1\bigr)^{\frac{q}{2}}+\bigl(3\beta_2\bigr)^{\frac{q}{2}}\Bigr)<1$,
we obtain, by Gronwall's Lemma as in \cite[Lemma~3.3]{KMRF12averaging},
$$\CCO{\expect\lVert X^n(t) - X^\infty(t)\rVert^2}^{\frac{q}{2}}\leq \frac{1}{3}\CCO{3\alpha_n(t)}^{\frac{q}{2}} +
\frac{1}{3}\int_{-\infty}^te^{-\gamma(t-s)}\bigl(3\alpha_n(s)\bigr)^{\frac{q}{2}}ds.$$
Using the dominated convergence theorem and the convergence of $\alpha_n(t)$ to $0$ as $n\rightarrow \infty$, we obtain the convergence
of $X^n(t)$ to $X^\infty(t)$  in quadratic mean. Hence  $X^n(t)$ converges in distribution to $X^\infty(t)$. But, since the
 distribution of  $X^n(t)$ is the same as that of $X(t+\gamma_n)$, we deduce that, for every $t\in \R$
$$\lim_{n\rightarrow\infty}\law{X(t+\alpha_n+\beta_n )} = \law{X^\infty(t)}.$$
Using the same reasoning as above and taking into account \eqref{eq:APF} and \eqref{eq:APG}, we can easily deduce that
$$\lim_{m\rightarrow\infty}\lim_{n\rightarrow\infty}\law{X(t+\alpha_m+\beta_n )} = \law{X^\infty(t)}.$$
Thus the solution $X$ is almost periodic in one-dimensional distribution.

To prove almost periodicity in $2$-UI distribution of the solution to \eqref{eq:SDE}, we need a generalization of \cite[Proposition~3.1]{DaPrato-Tudor95} to the Stepanov context. This allows us to obtain the convergence of the solutions by assuming only the convergence in mean (in Stepanov sense) of the coefficients. Let us mention that a similar result is obtained by  Ivo Vrko{\v{c}} \cite{Vrkoc_95}, but in another context.

\begin{prop}{\label{prop:conv en dist}}
Let $\tau\in\R$.
Let $(\xi_n)_n \subset \ellp{2}(\prob,
\h_2)$ be a sequence of random variables.
Let $F_n:\R\times\h_2\rightarrow \h_2$ and $G_n:\R\times\h_2\rightarrow \espLH(\h_1, \h_2)$, $n\in \N$, be two $\St^{2}$-bounded sequences. Assume that for each $n\in\N$,  the mappings $F_n$ and $G_n$ satisfy
${\rm(H2)}-{\rm(H3)}$, such that the constant $M$  is
independent of $n$ and the set of mappings $\{K_{n}, n\in\N\}$ is $\St^{2}$-bounded, that is, $\sup_{n\in \N}\norm{K_n}_{\St^2}<+\infty$.
Let $X_n$ be the unique mild solution to
\begin{equation*}
X_n(t)=T(t-\tau)\xi_n
+\int^{t}_{\tau}T(t-s)F_n\bigl(s, X_n(s)\bigl)ds +
\int^{t}_{\tau}T(t-s)G_n\bigl(s, X_n(s)\bigl)dW(s),\, t\geq\tau
\end{equation*}
in the space $\CUB\bigl(\R,\ellp{2}(\prob,
\h_2)\bigl)$.
Assume that for each $x$ in $\h_{2}$, the sequences $(F_n(.,x))_n$ and $(G_n(.,x))_n$ converge in Stepanov sense to $F_{\infty}(.,x)$ and $G_{\infty}(.,x)$ respectively, that is,
\begin{gather*}
\lim_{n\rightarrow\infty}\lVert F_n(.,x)-F_{\infty}(.,x)\rVert_{\St^{2}}=0,\
\lim_{n\rightarrow\infty}\lVert G_n(.,x)-G_{\infty}(.,x)\rVert_{\St^{2}}=0,
\end{gather*}
for each $x$ in $\h_{2}$.
It follows that
\begin{itemize}
  \item [a)] There exists a unique mild solution $X_\infty$ to
  \begin{equation}\label{eq:Xinfty}
X_\infty(t)=T(t-\tau)\xi_\infty
+\int^{t}_{\tau}T(t-s)F_\infty\bigl(s, X_\infty(s)\bigl)ds +
\int^{t}_{\tau}T(t-s)G_\infty\bigl(s, X_\infty(s)\bigl)dW(s),\, t\geq\tau.
  \end{equation}
  \item [b)]If\ $\lim_{n\rightarrow\infty}\expect\lVert \xi_n -
    \xi_\infty\rVert^2=0$, then, for all $\sigma\geq \tau$,
  \begin{equation}\label{eq:conv_Prop_APD_Step}
  \lim_{n\rightarrow\infty}\expect \CCO{\sup_{\tau\leq t\leq
      \sigma}\lVert X_n(t) - X_\infty(t)\rVert^2}=0.
  \end{equation}
  \item [c)] If\begin{gather*}
\lim_{n\rightarrow\infty}d_{\bl}(\law{\xi_n},\law{\xi_\infty})=0,
\end{gather*} then we have in $\Cont([\tau,\sigma]; \h_2)$,
for all $\sigma\geq\tau$,
\begin{equation}\label{eq:lim_dBL}
\lim_{n\rightarrow\infty}d_{\bl}(\law{X_n},\law{X_\infty})=0.
\end{equation}
\end{itemize}
\end{prop}
\begin{proof}
a) By Proposition \ref{lem:limits_inversion}, $F_\infty$ and
$G_\infty$ satisfy Conditions ${\rm(H2)}$, ${\rm(H3)}$, and
${\rm(H4)}$. We deduce a) as in the first step of the proof of Theorem
\ref{theo:main}.

\medskip\noindent
b) For any subsequence $(X'_n)$ of $(X_n)$, we can find, by
Proposition \ref{lem:limits_inversion}, a subsequence $(X''_n)$ of
$(X'_n)$ and versions $F''_\infty$ and $G''_\infty$ of $F_\infty$ and
$G_\infty$ respectively (i.e., $F''_\infty(t,.)=F_\infty(t,.)$ and
$G''_\infty(t,.)=G_\infty(t,.)$ for almost every $t$), such that the
corresponding subsequences
$(F''_n)$ and $(G''_n)$ converge pointwise to $F''_\infty$ and
$G''_\infty$ respectively.
Since the integrals in \eqref{eq:Xinfty} remain unchanged if we
replace $F_\infty$ by  $F''_\infty$ and $G_\infty$ by $G''_\infty$,
we deduce by \cite[Proposition~3.1]{DaPrato-Tudor95} that
\begin{equation}\label{eq:conv_Prop_APD_Step-subs}
  \lim_{n\rightarrow\infty}\expect \CCO{\sup_{\tau\leq t\leq
      \sigma}\lVert X''_n(t) - X_\infty(t)\rVert^2}=0.
  \end{equation}
Thus,
for any subsequence $(X'_n)$ of $(X_n)$ we can find a subsequence $(X''_n)$ of
$(X'_n)$ such that \eqref{eq:conv_Prop_APD_Step-subs} holds, which
proves \eqref{eq:conv_Prop_APD_Step}.

\medskip\noindent
c) Similarly, using \cite[Proposition~3.1]{DaPrato-Tudor95}, we obtain
that, for any subsequence $(X'_n)$ of $(X_n)$ we can find a subsequence $(X''_n)$ of
$(X'_n)$ such that
\begin{equation*}
\lim_{n\rightarrow\infty}d_{\bl}(\law{X''_n},\law{X_\infty})=0,
\end{equation*}
thus \eqref{eq:lim_dBL} holds.
\end{proof}
\finpr

\medskip
\proofof{Theorem \ref{theo:main} (continued)}

\medskip
\noindent{\textbf{Step~3: almost periodicity in $2$-UI distribution of the solution.}} To prove that $X$ is almost periodic in distribution, we use the same arguments as in \cite{KMRF12averaging} (see also \cite{bed-chal-mel-prf-sma2015}), using Proposition~\ref{prop:conv en dist}. For the uniform integrability part, we proceed as in \cite[page 1144]{bed-chal-mel-prf-sma2015}. We omit the details. 
 We can get more. From \eqref{eq:conv_Prop_APD_Step}, the sequence $\CCO{\norm{\tilde X(t+\gamma_n)}^2_{C_k}}$ is uniformly integrable. This means that $X \in \WASS^p_{[a,b]}\!\APD(\R,\h_2)$.
\finpr
\subsection{$\mu$-Pseudo almost periodicity of the solution in $2$-UI distribution}
Let $\mu$ be a Borel measure on $\R$ satisfying \eqref{eq:mu}
and Condition (\textbf{H}). Let us start  with a useful Lemma:
\begin{lemma}\label{lem:convolut}
Let $h\in\StER{q}(\R,\R,\mu)$, and let $K(.)$ be an $\St^{p}$-bounded function from $\R$ to $\R^+$. The function
$$t\mapsto\biggl(\int_{-\infty}^t e^{-2\delta(t-s)}K^2(s)h^2(s)\,ds\biggr)^{1/2}$$
is in $\ER(\R,\R,\mu)$.
\end{lemma}
\begin{proof}
Our proof uses the following result of \cite[Theorem 3.5]{blot-cieutat_ezzinbi2012}, that ensures that $\StER{q}(\R,\R,\mu)$ is  translation invariant. We have, for every $u\in\R$,
\begin{equation}\label{eq:invariance_h_ergodic}
\lim_{r\rightarrow+\infty}
 \frac{1}{\mu([-r,r])} \int_{[-r,r]}\biggl(\int_{0}^{1} \abs{h(t+u+s)}^{q}\ ds\biggr)^{\frac{1}{q}}\,d\mu(t)
=0.
\end{equation}
By Lebesgue's dominated convergence theorem and Hölder's inequality, we get
\begin{align*}
 &\frac{1}{\mu([-r,r])} \int_{[-r,r]} \biggl(\int_{-\infty}^t e^{-2\delta (t-s)}
           K^2(s) h^2(s)\, ds
        \biggr)^{1/2} d\mu(t)\\
&=  \frac{1}{\mu([-r,r])}
    \int_{[-r,r]}\biggl(\sum_{k=1}^{+\infty}\int_{0}^{1} e^{-2\delta (k-u)}
           K^2(t+u-k) h^2(t+u-k)\, du \biggr)^{1/2} d\mu(t)  \\
&\leq \lim_{n\rightarrow+\infty} \sum_{k=1}^{n} e^{-(k-1)p\delta}\norm{K}_{\St^{p}}^p\frac{1}{\mu([-r,r])}
   \int_{[-r,r]}\biggl(\int_{0}^{1} h^{q}(t+u-k)\, du \biggr)^{\frac{1}{q}} d\mu(t)  \\
&\leq \lim_{n\rightarrow+\infty} \sum_{k=1}^{n} e^{-(k-1)p\delta}\norm{K}_{\St^{p}}^p\frac{1}{\mu([-r,r])}
    \int_{[-r,r]}\biggl(\int_{0}^{1} h^{q}(t+u-k)\, du \biggr)^{\frac{1}{q}} d\mu(t).
\end{align*}
Since the
series
$$\sum_{k\geq 1} e^{-(k-1)p\delta}\norm{K}_{\St^{p}}^p\frac{1}{\mu([-r,r])}
    \int_{[-r,r]}\biggl(\int_{0}^{1} h^{q}(t+u-k)\, du \biggr)^{\frac{1}{q}} d\mu(t)$$
is uniformly convergent with respect to $r$, the claimed result is a consequence of \eqref{eq:invariance_h_ergodic}.
\end{proof}
\finpr
Before presenting the main result of this subsection, namely, the existence of $\mu$-pseudo almost periodic solution to Eq.~\eqref{eq:SDE}, let us introduce the following condition  $\rm(H4)'$ and some notations:
\begin{enumerate}
\item[$\rm(H4)'$] \label{cond2:Sppp of f and g}
$F\in \StPAP{2}(\R \times \h_2,\h_2,\mu)$ and $G \in \StPAP{2}( \R \times \h_2,\espLH(\h_1, \h_2),\mu) $.
\end{enumerate}
We denote by $(F_1,G_1)$ and $(F_2,G_2)$ the decompositions of $F$ and $G$ respectively, that is,
\begin{gather*}
F=F_1+F_2, \quad G=G_1+G_2, \\
F_1\in \StAP{2}(\R \times \h_2,\h_2),\
F_2\in \StER{2}(\R\times \h_2, \h_2,\mu),\\
G_1\in \StAP{2}( \R \times \h_2,\espLH(\h_1, \h_2)),\
G_2\in \StER{2}(\R\times \h_2, \espLH(\h_1, \h_2),\mu).
\end{gather*}
Let us now state the main result of the subsection.
\begin{theo}\label{theo:main3} Let $\mu$ be a Borel measure on $\R$ satisfying \eqref{eq:mu}
and Condition (\textbf{H}). Assume that $\rm{(H1)-(H3)}$ and $\rm{(H4)'}$  hold for both $F$ and $G$. Assume in addition that $F_1$ and $G_1$ satisfy the same growth and Lipschitz conditions $\rm{(H2)}$ and $\rm{(H3)}$  as $F$ and $G$
respectively, with same coefficient $M$ and mapping $K(.)$. Then:
\begin{enumerate}
  \item  Eq. \eqref{eq:SDE} admits a unique mild solution, $X$, in the space
$\CUB\bigl(\R,\ellp{2}(\prob, \h_2)\bigl)$ and $X$ has a.e.~continuous
trajectories, provided $\theta_{\St}<1$.
  \item Moreover, if $\theta'_{\St}<1$, then $X$ is $\mu$-pseudo almost periodic in $2$-UI distribution.
\end{enumerate}
\end{theo}
The proof of this theorem is inspired from \cite[Theorem 4.4]{bed-chal-mel-prf-sma2015}. Before giving it, let us first revisit the hypotheses in \cite[Theorems~4.4 and 4.3 ]{bed-chal-mel-prf-sma2015}  in the context of ($\mu$-pseudo) almost periodicity (resp.~$\mu$-pseudo almost automorphy).
\begin{remark}\begin{enumerate}{\em
 \item It should be noted that assuming that $F_1$ and $G_1$ satisfy the same growth and Lipschitz conditions as $F$ and $G$ respectively, as was done in \cite[Theorem~4.4]{bed-chal-mel-prf-sma2015}, is not necessary since these properties can be deduced from $\rm{(H2)}$ and $\rm{(H3)}$ imposed on both $F$ and $G$. Indeed, assume for instance that $F$ is $L$-Lipchitz (the reasoning is similar for the growth condition). Set, for all fixed $x,y \in \h_1$, and $t\in \R$, $\hat F(t,x,y):=\norm{F(t,x)-F(t,y)}_{\h_2}$. From the following rewriting of $\hat F(t,x,y)$: $\hat F(t,x,y)=\hat F_1(t,x,y)+ \CCO{\hat F(t,x,y)-\hat F_1(t,x,y)}$, and the $\mu$-ergodicity of the mapping
\begin{equation*}
\left[t\mapsto H(t,x,y):=\hat F(t,x,y)-\hat F_1(t,x,y)\right],
\end{equation*} we have $\hat F(.,x,y)\in \PAP(\R,\h_2,\mu)$. In view of the uniqueness of the previous decomposition (under Condition (\textbf{H})), we deduce that $$\accol{\hat F_1(t,x,y),\,t\in \R}\subset\overline{\accol{\hat F(t,x,y),\,t\in \R}}$$ (the closure of the range of $\hat F$). Consequently, for all $t\in \R$, $$\hat F_1(t,x,y)\leq \sup_{t\in \R} \hat F(t,x,y)\leq L \norm{x-y}_{\h_2}$$ from which we conclude that $F_1$ is $L$-Lipschitz. The same conclusion holds for $G_1$.
  \item Compared to the uniformity imposed on the almost periodicity of the parametric functions in \cite[Theorem~4.7]{bed-chal-mel-prf-sma2015} and \cite[Theorem~3.1]{KMRF12averaging}, which occurs with respect to the second variable in bounded subsets of $\h_2$, that imposed here, namely $F_1$ and $G_1$ are (Stepanov) almost periodic with respect to the second variable in compact subsets of $\h_2$, is clearly weaker.}
\end{enumerate}
\end{remark}

\proofof{Theorem \ref{theo:main3}}
The existence, uniqueness and properties of the mild solution $X$ to \eqref{eq:SDE} can be obtained using the same arguments as in  Theorem
\ref{theo:main}, that is, the classical method of the fixed point theorem for the
contractive operator $\Gamma$ on
$\CUB\bigl(\R,\ellp{2}(\prob, \h_2)\bigl)$
defined by
$$\Gamma X(t) = \int^{t}_{-\infty}T(t-s)F\bigl(s, X(s)\bigl)ds +
\int^{t}_{-\infty}T(t-s)G\bigl(s, X(s)\bigl)dW(s).$$

To show that $X$ is almost periodic in $2$-UI distribution, and therefore $X$ has a decomposition $X = Y + Z$ with $Y\in \APD^2(\R,\h_2)$  and $
Z\in\ER(\R,\ellp{2}(\esprob,\prob,\h_2),\mu)$, we proceed as in \cite{bed-chal-mel-prf-sma2015}. Let
$Y\in\CUB\bigl(\R,\ellp{2}(\prob, \h_2)\bigl)$ be the unique
almost periodic in $2$-UI distribution
mild solution to
\begin{equation}\label{eq:edsY}
dY(t)=
AY(t)\,dt + F_1(t, Y(t))\,dt + G_1(t, Y(t))\,dW(t), \ t\in\R.
\end{equation}
The existence and properties of $Y$ are guaranteed by Theorem
\ref{theo:main}.
The process $Y$ is thus the component almost periodic in $2$-UI distribution of $X$. To construct the $\mu$-ergodic component, namely, $Z$, we exploit the fact that  $X$, defined by \eqref{eq:mildsol}, is the limit in $\CUB\bigl(\R,\ellp{2}(\prob, \h_2)\bigl)$
of a sequence $(X_n)$ with  arbitrary $X_0$ and, for every $n$, $X_{n+1}=\Gamma(X_n)$.
Let us take a particular sequence, namely
$$
X_0=Y, \
X_{n+1}=\Gamma(X_n),\ Z_n=X_n-Y,\ n\in\N.
$$
(The first term is provided by the solution to Eq. \eqref{eq:edsY}). Let us prove by induction that each $Z_n$ is in
$\ER\bigl(\R,\ellp{2}(\prob,\h_2),\mu\bigl)$.  
We use some arguments of the proof of
\cite[Theorem 5.7]{blot-cieutat_ezzinbi2012}.

We have, for every $n\in\N$ and every $t\in\R$,
\begin{align*}
Z_{n+1}(t)=&\Gamma (X_n)(t)-Y(t)\\
=&\int_{-\infty}^t T(t-s)\bigl( F(s,X_n(s))-F(s,Y(s))\bigr)\,ds\\
&+\int_{-\infty}^t T(t-s)\bigl( G(s,X_n(s))-G(s,Y(s))\bigr)\,dW(s)\\
&+\int_{-\infty}^t T(t-s)\bigl( F(s,Y(s))-F_1(s,Y(s))\bigr)\,ds\\
&+\int_{-\infty}^t T(t-s)\bigl( G(s,Y(s))-G_1(s,Y(s))\bigr)\,dW(s)\\
=&\int_{-\infty}^t T(t-s)\bigl( F(s,X_n(s))-F(s,Y(s))\bigr)\,ds\\
&+\int_{-\infty}^t T(t-s)\bigl( G(s,X_n(s))-G(s,Y(s))\bigr)\,dW(s)\\
&+\int_{-\infty}^t T(t-s)F_2(s,Y(s))\,ds +\int_{-\infty}^t T(t-s)G_2(s,Y(s))\,dW(s)\\
=&J_1(t)+J_2(t)+J_3(t).
\end{align*}
Assume that $Z_{n}\in\ER\bigl(\R,\ellp{2}(\prob,\h_2),\mu\bigl)$. Since $\ER\bigl(\R,\ellp{2}(\prob,\h_2),\mu\bigl)\subset
\StER{q}\bigl(\R,\ellp{2}(\prob,\h_2),\mu\bigl)$, we have by the Lipschitz condition $\rm(H3)$,
$$\bigl(\expect\norm{F(s,X_n(s))-F(s,Y(s)}^2\bigr)^{1/2}
\leq K(s)\bigl(\expect\norm{Z_n(s)}^2\bigr)^{1/2}, \, \forall s\in \R.$$
The same inequality holds for $G$. Thus, using Hölder's inequality and Lemma \ref{lem:convolut}, we get
\begin{multline*}
\frac{1}{\mu([-r,r])}\int_{[-r,r]}
  \biggl(\expect \norm{J_1(s)}^2
     \biggr)^{1/2} d\mu(t)\\
     \begin{aligned}
     &\leq \frac{1}{\mu([-r,r])}\int_{[-r,r]}
  \biggl(\expect \norm{
    {\int_{-\infty}^t T(t-s)\bigl( F(s,X_n(s))-F(s,Y(s))\bigr)\,ds}}^2
     \biggr)^{1/2} d\mu(t)\\
&\leq  \frac{1}{(\delta)^{1/2}}\frac{1}{\mu([-r,r])}\int_{[-r,r]}
           \CCO{\int_{-\infty}^t e^{-\delta (t-s)}K^{2}(s) \expect\norm{Z_n(s)}^2\,ds
        }^{1/2}d\mu(t)\\
&\rightarrow 0\text{ when }r\rightarrow +\infty,
\end{aligned}
\end{multline*}
and
\begin{multline*}\frac{1}{\mu([-r,r])}\int_{[-r,r]}
  \biggl(\expect \norm{J_2(s)}^2
     \biggr)^{1/2} d\mu(t)\\
     \begin{aligned}
     &\leq\frac{1}{\mu([-r,r])}\int_{[-r,r]}
  \biggl(\expect\norm{
    {\int_{-\infty}^t T(t-s)\bigl( G(s,X_n(s))-G(s,Y(s))\bigr)\,dW(s)}}^2
     \biggr)^{1/2} d\mu(t)\\
&\leq  (\trace Q)^{1/2}\frac{1}{\mu([-r,r])}\int_{[-r,r]}
           \CCO{\int_{-\infty}^t e^{-2\delta (t-s)} K^{2}(s) \expect\norm{Z_n(s)}^2\,ds
        }^{1/2}d\mu(t)\\
&\rightarrow 0\text{ when }r\rightarrow +\infty.
\end{aligned}
\end{multline*}
Hence, $J_1(.)$ and $J_2(.)$  are in $\ER\bigl(\R,\ellp{2}(\prob,\h_2),\mu\bigl)$.
Now, by using the same reasoning as in \cite[page 36]{bed-chal-mel-prf-sma2015}, let us prove that
$J_3(.)$  is in
$\ER\bigl(\R,\ellp{2}(\prob,\h_2),\mu\bigl)$.
The almost periodicity in distribution of $Y$ implies that the family
$({Y}(u+.))_{u\in\R}$ is uniformly tight in
$C_k(\R,\h_2)$. Therefore, for each $\epsilon>0$
there exists a compact subset $\mathcal{K}_\epsilon$ of
$C_k(\R,\h_2)$ such that, for every $u\in \R$,
$$\prob\accol{Y(u+.)\in\mathcal{K}_\epsilon}\geq 1-\epsilon.$$
Then, there exists a compact subset $K_\epsilon$ of $\h_2$ such that, for every $u,t\in \R$,
\begin{equation*}
\prob\accol{(\forall s\in[t,t+1]);\ Y(u+s)\in K_{\epsilon}}\geq 1-\epsilon.
\end{equation*}
For $u=0$, we obtain,
\begin{equation}\label{eq:Ytendu}
\prob\accol{(\forall s\in[t,t+1]);\ Y(s)\in K_{\epsilon}}\geq 1-\epsilon.
\end{equation}
Let $\Omega_{\epsilon,t}$ be the
measurable subset of $\Omega$ on which \eqref{eq:Ytendu} holds.
By compactness of $K_{\epsilon}$,
we can find a finite sequence
$y_1,\dots,y_{n(\epsilon)}$ such that
$$K_{\epsilon} \subset \cup_{i=1}^{n(\epsilon)}B(y_i, \epsilon), $$
and we get, using \eqref{eq:Ytendu}, for every $t\in \R$,
\begin{equation}\label{eq:uniform}
 \sup_{s\in [t,t+1]}\expect\biggl( \min_{1\leq i\leq
  n(\epsilon)}\bigl(1_{\Omega_{\epsilon,t}}\norm{Y(s)-y_i}^2\bigr)\biggr)<\epsilon.
\end{equation}
It is straightforward to see that $F_2=F-F_1$ and $G_2=G-G_1$ satisfy conditions similar to $\rm(H2)$ and $\rm(H3)$.
We have then by Itô's isometry
\begin{multline*}
\frac{1}{\mu([-r,r])}\int_{[-r,r]}
  \expect\biggl\Vert\int_{-\infty}^t T(t-s)F_2(s,Y(s))ds+\int_{-\infty}^t T(t-s)G_2(s,Y(s))\,dW(s)\biggr\Vert^2 d\mu(t)\\
\begin{aligned}
\leq&\frac{2\delta^{-1}}{\mu([-r,r])}\int_{[-r,r]}
 \int_{-\infty}^t e^{-\delta (t-s)}\expect\norm{F_2(s,Y(s))}^2 ds
        d\mu(t) \\
    &+\frac{2\trace Q}{\mu([-r,r])}\int_{[-r,r]}
  \int_{-\infty}^t
         e^{ -2\delta (t-s)}\expect\norm{G_2(s,Y(s))}^2\,dsd\mu(t)\\
        =& J_3^1(r)+J_3^2(r).
\end{aligned}
\end{multline*}
Let us deal with the term $J_3^1(r)$. We have
 \begin{multline*}
\begin{aligned}
         J_3^1(r) \leq&\frac{2\delta^{-1}}{\mu([-r,r])}\int_{[-r,r]}
 \CCO{\sum_{k=1}^{+\infty}\int_{t-k}^{t-k+1}e^{-\delta (t-s)}\expect\norm{F_2(s,Y(s))}^2 ds
        }d\mu(t)\\
\leq&\frac{4\delta^{-1}}{\mu([-r,r])}\int_{[-r,r]}
\sum_{k=1}^{+\infty}e^{-\delta (k-1)}\int_{t-k}^{t-k+1}
\expect\biggl(\min_{1\leq i \leq n}\bigl(\un{\Omega_{\epsilon,t}}\norm{F_2(s,Y(s))-F_2(s,y_{i})}^2\bigr)\biggr) ds
        d\mu(t) \\
    &+\max_{1\leq i \leq n}\frac{4\delta^{-1}}{\mu([-r,r])}\int_{[-r,r]}\CCO{\int_{-\infty}^t e^{-\delta (t-s)}
    \expect\big(\norm{F_2(s,y_{i})}^2\big) ds }d\mu(t) \\
    &+\frac{4\delta^{-1}}{\mu([-r,r])}\int_{[-r,r]}\CCO{\int_{-\infty}^t e^{-\delta (t-s)}
    \expect\big(\un{\Omega_{\epsilon,t}^{c}}\norm{F_2(s,Y(s))}^2\big) ds }d\mu(t) \\
    = & I_{1}(r)+I_{2}(r)+I_{3}(r).
\end{aligned}
\end{multline*}
Using the Lipschitz condition $\rm(H3)$ and the estimation \eqref{eq:uniform}, we obtain
\begin{align*}
  I_{1}(r) 
    &\leq \frac{4\delta^{-1}}{\mu([-r,r])}\int_{[-r,r]}
\CCO{\sum_{k=1}^{+\infty}e^{-\delta (k-1)}\sup_{t\in \R}\int_{t}^{t+1}K^2(s)
\expect\biggl(\min_{1\leq i \leq n}\bigl(\un{\Omega_{\epsilon,t}}\norm{Y(s))-y_{i}}^2\bigr)\biggr) ds
        }d\mu(t) \\
    &\leq \frac{4\delta^{-1}\norm{K}^2_{\St^2}}{1-e^{-\delta}}\epsilon.
\end{align*}
Thanks to the ergodicity of $F_2$ and Lebesgue's dominated convergence theorem, one obtains,
for any $r>0$,
\begin{align*}
  I_{2}(r) &= \max_{1\leq i \leq n}\frac{4\delta^{-1}}{\mu([-r,r])}\int_{[-r,r]}\CCO{\int_{-\infty}^t e^{-\delta (t-s)}
    \expect\big(\norm{F_2(s,y_{i})}^2\big) ds }d\mu(t) \\
    &\leq \max_{1\leq i \leq n}\frac{4\delta^{-1}}{\mu([-r,r])}\int_{[-r,r]}
\CCO{\sum_{k=1}^{+\infty}e^{-\delta (k-1)}\int_{t-k}^{t-k+1}
\expect\bigl(\norm{F_2(s,y_{i})}^2\bigr) ds
        }d\mu(t)  \\
    &= \max_{1\leq i \leq n}\frac{4\delta^{-1}}{\mu([-r,r])}\int_{[-r,r]}
\CCO{\sum_{k=1}^{+\infty}e^{-\delta (k-1)}\int_{t-k}^{t-k+1}
\expect\bigl(\norm{F_2(s,y_{i})}^2\bigr) ds
        }d\mu(t)  \\
    &= \sum_{k=1}^{+\infty}e^{-\delta (k-1)}\max_{1\leq i \leq n}\frac{4\delta^{-1}}{\mu([-r,r])}\int_{[-r,r]}
\CCO{\int_{0}^{1}
\expect\bigl(\norm{F_2(t-k +s,y_{i})}^2\bigr) ds
        }d\mu(t).
\end{align*}
Arguing as in the proof of Lemma~\ref{lem:convolut}, we deduce that $\lim_{r\rightarrow\infty}I_{2}(r)=0$.
On the other hand, by  Condition $\rm(H2)$ and the uniform integrability of the family $(\norm{Y(t)}^2)_{t\in \R}$, we get
\begin{align*}
  I_{3}(r) &= \frac{4\delta^{-1}}{\mu([-r,r])}\int_{[-r,r]}\CCO{\int_{-\infty}^t e^{-\delta (t-s)}
    \expect\big(\un{\Omega_{\epsilon,t}^{c}}\norm{F_2(s,Y(s))}^2\big) ds }d\mu(t)\\
    &\leq  \frac{4M\delta^{-1}}{\mu([-r,r])}\int_{[-r,r]}\CCO{\int_{-\infty}^t e^{-\delta (t-s)}
    \expect\big(\un{\Omega_{(\epsilon,t)}^{c}}(1+\norm{Y(s))})^2\big) ds }d\mu(t) \\
    &\leq \frac{8M\delta^{-1}}{\mu([-r,r])}\int_{[-r,r]}\CCO{\sum_{k=1}^{+\infty}e^{-\delta (k-1)}
   \biggl( \prob\CCO{\Omega_{\epsilon,t}^{c}}+\int_{t-k}^{t-k+1}\expect\bigl(\un{\Omega_{\epsilon,t}^{c}}
   \norm{Y(s))}^2\bigr) ds\biggr) }d\mu(t) \\
    &\leq \frac{8M\delta^{-1}\epsilon}{1-e^{-\delta}}.
\end{align*}
As $\epsilon$ is arbitrary, we  deduce from the previous estimations on $I_{1}(r),I_{2}(r)$ and $I_{3}(r)$, that $\lim_{r\rightarrow\infty}J_{3}^1(r)=0$.
In the same way, we  can estimate the term $J_{3}^2(r)$. We then easily see that $\lim_{r\rightarrow\infty}J_{3}^2(r)=0$. Which shows that
$J_3(.)$  is in $\ER\bigl(\R,\ellp{2}(\prob,\h_2),\mu\bigl)$.
Consequently, $Z_{n+1}\in \ER\bigl(\R,\ellp{2}(\prob,\h_2),\mu\bigl)$.

So we have shown that the sequence  $(Z_n)$
lies in
$\ER\bigl(\R,\ellp{2}(\prob,\h_2),\mu\bigl)$.

Now, the sequence $(X_n)$ converges to $X$ in
$\CUB\bigl(\R,\ellp{2}(\prob, \h_2)\bigl)$,
thus
$(Z_n)$ converges to $Z:=X-Y$ in
$\CUB\bigl(\R,\ellp{2}(\prob, \h_2)\bigl)$. Using the same calculations as in \cite[page 1150]{bed-chal-mel-prf-sma2015}, we obtain that
$Z\in\ER\bigl(\R,\ellp{2}(\prob,\h_2),\mu\bigl)$.
\finpr
\section{Comments and concluding remarks}\label{sec:conslusion}
When the function $g$ in \eqref{eq:SDE} is equal to zero, we retrieve a semilinear (deterministic) differential equation in the Banach space $\ellp{2}(\prob,\h_2)$:
\begin{equation}\label{eq:EDA_g=0}
    u'(t) = Au(t) + f(t, u(t)),\, t \in \R.
\end{equation}
An extensive literature (see e.g. \cite{Ding-Long-N'guerekata2011,long_composition_2011,li_stepanov-like_2011,zhao_new_2011}), is devoted to the problem of existence and uniqueness of a bounded ($\mu$-pseudo) almost periodic mild solution to \eqref{eq:EDA_g=0} in a Banach space $\espX$.  The adopted approach is based on superposition theorems in the Banach space  $\StAP{p}(\R,\espX)$ (or $\StPAP{p}(\R,\espX,\mu)$) combined with the Banach's fixed-point principle, applied to the nonlinear operator $$(\Gamma u)(t)=\int_{-\infty}^t T(t-s)f(s,u(s))ds.$$
To our knowledge,  all existing results use the fact that $\Gamma$
maps $\AP(\R,\espX)$ into itself, but $\Gamma$ does not map
  $\StAP{p}(\R,\espX)$ into $\AP(\R,\espX)$ nor into
  $\StAP{p}(\R,\espX)$ (see in
  particular~\cite{Ding-Long-N'guerekata2011,long_composition_2011,li_stepanov-like_2011,zhao_new_2011}). The
  proposed proofs may be summarized as follows: if $u\in
  \AP(\R,\espX)$, then $u$ satisfies the compactness condition (Com)
  of Subsection \ref{subsec:superposition}, and $u\in
  \StAP{p}(\R,\espX)$. From the existing superposition theorems (see
  e.g.~\cite[Theorem 2.1]{Ding-Long-N'guerekata2011}) combined with
  Condition (Lip) of Subsection
    \ref{subsec:superposition},
  it follows that $F(.):=f(.,u(.))\in
  \StAP{p}(\R,\espX)$, and then $(\Gamma
  u)(.)=\int_{-\infty}^. T(.-s)F(s)ds \in \AP(\R,\espX)$. This
obviously shows the existence (and uniqueness) of an almost periodic
mild solution to \eqref{eq:EDA_g=0}, but it does not exclude the
possibility of existence of a purely Stepanov almost periodic
solution. The main difficulty in showing the nonexistence of a purely
Stepanov almost periodic bounded solution with the tools used in the
literature (see for example \cite{Ding-Long-N'guerekata2011}), arises
from the imposed compactness condition (Com) in the superposition
theorem of Stepanov almost periodic functions, which
  seems strong enough in $\StAP{p}(\R,\espX)$.  Thanks to Theorem
\ref{thm:theorem-composition}, it is easy to see that under Condition
(Lip), the operator $\Gamma$ maps $\StAP{p}(\R,\espX)$ into $\AP(\R,\espX)$. This shows that the obtained bounded solution cannot be purely Stepanov almost periodic.

Now, we focus on the following problem: can we expect a
similar conclusion if we replace the assumption that $f$ is Stepanov
almost periodic by the assumption that $f$ is Stepanov almost periodic in
Lebesgue measure?  Obviously, the answer depends on the
Bohl-Bohr-Amerio Theorem, see e.g.~\cite[p.~80]{Levitan_Zhikov82}, for
functions which are Stepanov almost periodic in Lebesgue measure.  What we need
is to solve first another problem, namely, does the boundedness of the
indefinite integral of a function which is Stepanov almost periodic in Lebesgue measure imply its Bohr almost-periodicity? The answer to this question is negative as shown by the following example:
\begin{example}\label{exple:Bohl-Bohr-Thm} {\em It is well-known that the Levitan function $H:\R\rightarrow\R$, defined by  $$H(t)= \sin\CCO{\frac{1}{g(t)}},$$
where $$g(t)=2+\cos(t)+\cos(\sqrt{2}t)$$
is a bounded $\St^{p}$-almost periodic function but not Bohr-almost periodic (see e.g. \cite{Levitan_Zhikov82}). Let's denote by $h$ its derivative. We have
\begin{equation}\label{eq:h}
  h(t)=\cos\CCO{\frac{1}{g(t)}}\biggl(\frac{\sin(t)+\sqrt{2}\sin(\sqrt{2}t)}{g^{2}(t)}\biggr).
\end{equation}
By the Bohl-Bohr Theorem, the derivative function $h$ cannot be Stepanov almost periodic (see \cite{andres_stepanov_2012}), and consequently, the function $1/g$ cannot be Stepanov almost periodic. But one can easily observe by taking into account 2.~in Remark~\ref{rem:metric property of S_D}, that $h$ is in $\StAP{0}(\R)$ as a product of $\St^{0}$-almost periodic functions.}
\end{example}

In order to answer the first problem, we  give, in the following,  a simple affine scalar equation
with purely Stepanov almost periodic in Lebesque measure coefficient. We see that the unique  bounded solution is not Bohr-almost periodic (but it is purely Stepanov almost periodic).

\begin{example}{\em
Consider the affine differential equation
\begin{equation}\label{deteministe edo}
x'(t)=-x(t)+h(t), t\in\R,
\end{equation}
where $h$ is given by \eqref{eq:h}. The unique bounded solution to \eqref{deteministe edo} is given by:
\begin{eqnarray*}
 x(t)=\int^{t}_{-\infty} e^{(s-t)}h(s)ds
   = \sin\CCO{\frac{1}{g(t)}}+ \int^{t}_{-\infty} e^{(s-t)}\sin\CCO{\frac{1}{g(s)}}ds.
\end{eqnarray*}
  The boundedness of $x$ follows from
$$|x(t)|\leq \abs{\sin\CCO{\frac{1}{g(t)}}}+ \int^{t}_{-\infty} e^{(s-t)}\abs{\sin\CCO{\frac{1}{g(s)}}}ds\leq 2,\ \forall t\in\R.$$
But $x$ is not Bohr-almost periodic, as it is the sum of the purely Stepanov almost periodic function $H(t)=\sin\CCO{\frac{1}{g(t)}}$ and a Bohr almost periodic function. } \end{example}

In comparison to the consequences obtained by Andres and Pennequin \cite[Consequence~1, p.~1667 and Consequence~4, p.~1679]{andres_stepanov_2012}, this example shows (in addition) that one can obtain the existence and the uniqueness of a bounded purely Stepanov almost periodic solution when the coefficients are purely Stepanov almost periodic in  Lebesgue measure.

This simple result can open new directions about the problem of existence of purely Stepanov almost periodic solutions, in both stochastic and deterministic cases.
%


\begin{thebibliography}{10}

\bibitem{Amerio-Prouse-1971-book}
Luigi Amerio and Giovanni Prouse.
\newblock {\em Almost-periodic functions and functional equations}.
\newblock Van Nostrand Reinhold Co., New York-Toronto, Ont.-Melbourne, 1971.

\bibitem{Andres-Bersani-Grande-06-Hierarchy}
J.~Andres, A.~M. Bersani, and R.~F. Grande.
\newblock Hierarchy of almost-periodic function spaces.
\newblock {\em Rend. Mat. Appl., VII. Ser.}, 26(2):121--188, 2006.

\bibitem{Andres99}
Jan Andres.
\newblock Almost-periodic and bounded solutions of {C}arath\'eodory
  differential inclusions.
\newblock {\em Differential Integral Equations}, 12(6):887--912, 1999.

\bibitem{andres_almost-periodicity_2001}
Jan Andres and Alberto~M. Bersani.
\newblock Almost-periodicity problem as a fixed-point problem for evolution
  inclusions.
\newblock {\em Topol. Methods Nonlinear Anal.}, 18(2):337--349, 2001.

\bibitem{Andres_Bersani_Lesniak_01_APvarious_metric}
Jan {Andres}, Alberto~M. {Bersani}, and Krzysztof {Le\'sniak}.
\newblock {On some almost-periodicity problems in various metrics.}
\newblock {\em {Acta Appl. Math.}}, 65(1-3):35--57, 2001.

\bibitem{andres_stepanov_2012}
Jan Andres and Denis Pennequin.
\newblock On {Stepanov} almost-periodic oscillations and their discretizations.
\newblock {\em J. Difference Equ. Appl.}, 18(10):1665--1682, 2012.

\bibitem{Andres-Pennequin2012}
Jan Andres and Denis Pennequin.
\newblock On the nonexistence of purely {Stepanov} almost-periodic solutions of
  ordinary differential equations.
\newblock {\em Proc. Am. Math. Soc.}, 140(8):2825--2834, 2012.

\bibitem{Pankov-1990-book}
A.~A.~Pankov (auth.).
\newblock {\em Bounded and Almost Periodic Solutions of Nonlinear Operator
  Differential Equations}.
\newblock Mathematics and its Applications 55. Springer Netherlands, 1 edition,
  1990.

\bibitem{bed-chal-mel-prf-sma2015}
Fazia Bedouhene, Nouredine Challali, Omar Mellah, Paul Raynaud~de Fitte, and
  Mannal Smaali.
\newblock Almost automorphy and various extensions for stochastic processes.
\newblock {\em J. Math. Anal. Appl.}, 429(2):1113--1152, 2015.

\bibitem{bedouhene-mellah-prf2012}
Fazia Bedouhene, Omar Mellah, and Paul Raynaud~de Fitte.
\newblock Bochner-almost periodicity for stochastic processes.
\newblock {\em Stoch. Anal. Appl.}, 30(2):322--342, 2012.

\bibitem{Besicovitch-1955-book}
Abram~Samoilovitch Besicovitch.
\newblock {\em Almost periodic functions}.
\newblock Dover Publications, 1954.

\bibitem{Bezandry-Diagana2008_Step_quad}
P.~Bezandry and T.~Diagana.
\newblock Existence of ${S}^2$-almost periodic solutions to a class of
  nonautonomous stochastic evolution equations.
\newblock {\em Electron. J. Qual. Theory Differ. Equ.}, 2008:19, 2008.

\bibitem{blot-mophou-nguerekata-pennequin2009}
J.~Blot, G.~M. Mophou, G.~M. N'Gu{\'e}r{\'e}kata, and D.~Pennequin.
\newblock Weighted pseudo almost automorphic functions and applications to
  abstract differential equations.
\newblock {\em Nonlinear Anal.}, 71(3-4):903--909, 2009.

\bibitem{blot-Cieutat16Completness}
Joel Blot and Philippe Cieutat.
\newblock Completeness of sums of subspaces of bounded functions and
  applications.
\newblock {\em Commun. Math. Anal.}, 19(2):43--61, 2016.

\bibitem{blot-cieutat_ezzinbi2012}
Jo{\"e}l Blot, Philippe Cieutat, and Khalil Ezzinbi.
\newblock Measure theory and pseudo almost automorphic functions: new
  developments and applications.
\newblock {\em Nonlinear Anal.}, 75(4):2426 -- 2447, 2012.

\bibitem{blot-cieutat-ezzinbi2013}
Jo{\"e}l Blot, Philippe Cieutat, and Khalil Ezzinbi.
\newblock New approach for weighted pseudo-almost periodic functions under the
  light of measure theory, basic results and applications.
\newblock {\em Appl. Anal.}, 92(3):493--526, 2013.

\bibitem{blot-al09superposition}
Jo{\"e}l Blot, Philippe Cieutat, Gaston~M. N'Gu{\'e}r{\'e}kata, and Denis
  Pennequin.
\newblock Superposition operators between various almost periodic function
  spaces and applications.
\newblock {\em Commun. Math. Anal.}, 6(1):42--70, 2009.

\bibitem{bochner33abstrakte}
S.~Bochner.
\newblock Abstrakte {F}astperiodische {F}unktionen.
\newblock {\em Acta Math.}, 61(1):149--184, 1933.

\bibitem{bochner62new_approach}
S.~Bochner.
\newblock A new approach to almost periodicity.
\newblock {\em Proc. Nat. Acad. Sci. U.S.A.}, 48:2039--2043, 1962.

\bibitem{chang_stepanov-like_2016}
Yong-Kui Chang, Zhuan-Xia Cheng, and Gaston~M. N'Gu{\'e}r{\'e}kata.
\newblock Stepanov-like pseudo almost automorphic solutions to some stochastic
  differential equations.
\newblock {\em Bull. Malays. Math. Sci. Soc. (2)}, 39(1):181--197, 2016.

\bibitem{chang_stepanov-like_2011}
Yong-Kui Chang, Zhi-Han Zhao, Gaston~M. N'Guerekata, and Ruyun Ma.
\newblock Stepanov-like almost automorphy for stochastic processes and
  applications to stochastic differential equations.
\newblock {\em Nonlinear Anal., Real World Appl.}, 12(2):1130--1139, 2011.

\bibitem{Corduneanu89book}
C.~{Corduneanu}.
\newblock {\em {Almost periodic functions. With the collaboration of N.
  Gheorghiu and V. Barbu. Transl. from the Romanian by Gitta Berstein and
  Eugene Tomer. 2nd Engl. ed.}}
\newblock New York: Chelsea Publishing Company, 2nd engl. ed. edition, 1989.

\bibitem{DaPrato-Tudor95}
G.~Da~Prato and C.~Tudor.
\newblock Periodic and almost periodic solutions for semilinear stochastic
  equations.
\newblock {\em Stochastic Anal. Appl.}, 13(1):13--33, 1995.

\bibitem{dapratozabczyk14book}
Giuseppe Da~Prato and Jerzy Zabczyk.
\newblock {\em Stochastic equations in infinite dimensions}, volume 152 of {\em
  Encyclopedia of Mathematics and its Applications}.
\newblock Cambridge University Press, Cambridge, second edition, 2014.

\bibitem{danilov_measure-valued_1997}
L.~I. Danilov.
\newblock Measure-valued almost periodic functions.
\newblock {\em Math. Notes}, 61(1):48--57, 1997.

\bibitem{Danilov-1997}
L.~I. Danilov.
\newblock Measure-valued almost periodic functions and almost periodic
  selections of multivalued mappings.
\newblock {\em Mat. Sb.}, 188(10):3--24, 1997.

\bibitem{Danilov98-uniform-approximation}
L.~I. Danilov.
\newblock On the uniform approximation of a function that is almost periodic in
  the sense of {S}tepanov.
\newblock {\em Izv. Vyssh. Uchebn. Zaved. Mat.}, 5:10--18, 1998.

\bibitem{danilov_almost_2000}
L.~I. Danilov.
\newblock Almost periodic measure-valued functions.
\newblock {\em Sb. Math.}, 191(12):1773--1796, 2000.

\bibitem{diagana06weighted}
Toka Diagana.
\newblock Weighted pseudo almost periodic functions and applications.
\newblock {\em C. R. Math. Acad. Sci. Paris}, 343(10):643--646, 2006.

\bibitem{diagana07book}
Toka Diagana.
\newblock {\em Pseudo almost periodic functions in {B}anach spaces}.
\newblock Nova Science Publishers, Inc., New York, 2007.

\bibitem{diagana_2008_SPA}
Toka Diagana.
\newblock Stepanov-like pseudo-almost periodicity and its applications to some
  nonautonomous differential equations.
\newblock {\em Nonlinear Anal., Theory Methods Appl., Ser. A, Theory Methods},
  69(12):4277--4285, 2008.

\bibitem{diagana08weighted}
Toka Diagana.
\newblock Weighted pseudo-almost periodic solutions to some differential
  equations.
\newblock {\em Nonlinear Anal.}, 68(8):2250--2260, 2008.

\bibitem{diagana09stepanov}
Toka Diagana.
\newblock Existence of pseudo-almost automorphic solutions to some abstract
  differential equations with {${S}^p$}-pseudo-almost automorphic coefficients.
\newblock {\em Nonlinear Anal.}, 70(11):3781--3790, 2009.

\bibitem{diagana_existence_2009_SAA}
Toka Diagana.
\newblock Existence of pseudo-almost automorphic solutions to some abstract
  differential equations with ${S}^p$-pseudo-almost automorphic coefficients.
\newblock {\em Nonlinear Anal., Theory Methods Appl., Ser. A, Theory Methods},
  70(11):3781--3790, 2009.

\bibitem{Ding-Long-N'guerekata2011}
Hui-Sheng Ding, Wei Long, and Gaston~M. N'Gu{\'e}r{\'e}kata.
\newblock Almost periodic solutions to abstract semilinear evolution equations
  with {Stepanov} almost periodic coefficients.
\newblock {\em J. Comput. Anal. Appl.}, 13(2):231--242, 2011.

\bibitem{Fan-Liang-Xiao2011}
Zhenbin Fan, Jin Liang, and Ti-Jun Xiao.
\newblock On {S}tepanov-like (pseudo) almost automorphic functions.
\newblock {\em Nonlinear Analysis: Theory, Methods \& Applications}, 74(8):2853
  -- 2861, 2011.

\bibitem{Fink74book}
A.~M. Fink.
\newblock {\em Almost periodic differential equations}.
\newblock Lecture Notes in Mathematics, Vol. 377. Springer-Verlag, Berlin,
  1974.

\bibitem{Franklin1929}
P.~Franklin.
\newblock Almost periodic recurrent motions.
\newblock {\em Mathematische Zeitschrift}, 30:325--331, 1929.

\bibitem{hu_stepanov-like_2012}
Zhanrong Hu and Zhen Jin.
\newblock Stepanov-like pseudo almost periodic mild solutions to nonautonomous
  neutral partial evolution equations.
\newblock {\em Nonlinear Anal., Theory Methods Appl., Ser. A, Theory Methods},
  75(1):244--252, 2012.

\bibitem{hu_boundeness_2005}
Zuosheng Hu.
\newblock Boundeness and {Stepanov}'s almost periodicity of solutions.
\newblock {\em Electron. J. Differ. Equ.}, 2005:7, 2005.

\bibitem{Hu-Mingarelli08}
Zuosheng Hu and Angelo~B. Mingarelli.
\newblock Bochner's theorem and {Stepanov} almost periodic functions.
\newblock {\em Ann. Mat. Pura Appl. (4)}, 187(4):719--736, 2008.

\bibitem{hurd_stepanov_1996}
H.~L. Hurd and A.~Russek.
\newblock Stepanov almost periodically correlated and almost periodically
  unitary processes.
\newblock {\em Theory Probab. Appl.}, 41(3):449--467 (1996) and teor. veroyatn.
  primen. 41, no. 3, 591--611, 1996.

\bibitem{KMRF12averaging}
Mikhail Kamenskii, Omar Mellah, and Paul Raynaud~de Fitte.
\newblock Weak averaging of semilinear stochastic differential equations with
  almost periodic coefficients.
\newblock {\em J. Math. Anal. Appl.}, 427(1):336--364, 2015.

\bibitem{komlos67gene}
J.~Komlós.
\newblock A generalization of a problem of {S}teinhaus.
\newblock {\em Acta Mathematica Hungarica}, 18(1-2):217--229, 1967.

\bibitem{Levitan53}
B.~M. {Levitan}.
\newblock {\em Almost periodic functions}.
\newblock Moskva: Gosudarstv. Izdat. Tehn.-Teor. Lit., 1953.

\bibitem{Levitan_Zhikov82}
B.~M. Levitan and V.~V. Zhikov.
\newblock {\em Almost periodic functions and differential equations}.
\newblock CUP Archive, 1982.

\bibitem{li_stepanov-like_2011}
Hong-Xu Li and Li-Li Zhang.
\newblock Stepanov-like pseudo-almost periodicity and semilinear differential
  equations with uniform continuity.
\newblock {\em Result. Math.}, 59(1-2):43--61, 2011.

\bibitem{long_composition_2011}
Wei Long and Hui-Sheng Ding.
\newblock Composition theorems of {Stepanov} almost periodic functions and
  {Stepanov}-like pseudo-almost periodic functions.
\newblock {\em Adv. Difference Equ.}, 2011:12, 2011.

\bibitem{MRF13}
Omar Mellah and Paul Raynaud~de Fitte.
\newblock Counterexamples to mean square almost periodicity of the solutions of
  some {SDE}s with almost periodic coefficients.
\newblock {\em Electron. J. Differential Equations}, pages No. 91, 7, 2013.

\bibitem{Nemytskii_Stepanov_book_60}
V.~V. Nemytskii and V.~V. Stepanov.
\newblock {\em Qualitative theory of differential equations}.
\newblock Princeton Mathematical Series, No. 22. Princeton University Press,
  Princeton, N.J., 1960.

\bibitem{nguerekata05book}
Gaston~M. N'Gu{\'e}r{\'e}kata.
\newblock {\em Topics in almost automorphy}.
\newblock Springer-Verlag, New York, 2005.

\bibitem{nguerekata-pankov08stepanov}
Gaston~M. N'Gu{\'e}r{\'e}kata and Alexander Pankov.
\newblock Stepanov-like almost automorphic functions and monotone evolution
  equations.
\newblock {\em Nonlinear Anal.}, 68(9):2658--2667, 2008.

\bibitem{KASPRZAK_16}
Justyna Signerska-Rynkowska Piotr~Kasprzak, Adam~Nawrocki.
\newblock {Integrate-and-fire models with an almost periodic input function}.
\newblock {\em ArXiv e-prints}, 1610.04434v1, October 2016.

\bibitem{rao_almost_1999}
Aribindi~Satyanarayan Rao.
\newblock On the almost periodic solution of a second-order infinitesimal
  generator differential equation.
\newblock {\em Bull. Calcutta Math. Soc.}, 91(5):391--396, 1999.

\bibitem{rao_higher-order_2004}
Aribindi~Satyanarayan Rao.
\newblock On a higher-order evolution equation with a {Stepanov}-bounded
  solution.
\newblock {\em Int. J. Math. Math. Sci.}, 2004(69-72):3959--3964, 2004.

\bibitem{Stepanoff1926}
W.~Stepanoff.
\newblock {\"U}ber einige verallgemeinerungen der fast periodischen funktionen.
\newblock {\em Mathematische Annalen}, 95(1):473--498, 1926.

\bibitem{Stoinski_AP_Lebesgue_measure94}
Stanis{\l}aw {Stoi\'nski}.
\newblock {Almost periodic functions in the Lebesgue measure.}
\newblock {\em {Ann. Soc. Math. Pol., Ser. I, Commentat. Math.}}, 34:189--198,
  1994.

\bibitem{stoinski_remarks_1996}
Stanis{\l}aw Stoi{\'n}ski.
\newblock Some remarks on {B}ohr's almost periodic functions and {S}tepanov's
  almost periodic functions.
\newblock {\em Funct. Approx. Comment. Math.}, 24:53--58, 1996.

\bibitem{Stoinski99_compatness}
Stanis{\l}aw {Stoi\'nski}.
\newblock {On compactness of almost periodic functions in the Lebesgue
  measure.}
\newblock {\em {Fasc. Math.}}, 30:171--175, 1999.

\bibitem{tang_stepanov-like_2014}
Chao Tang and Yong-Kui Chang.
\newblock Stepanov-like weighted asymptotic behavior of solutions to some
  stochastic differential equations in {Hilbert} spaces.
\newblock {\em Appl. Anal.}, 93(12):2625--2646, 2014.

\bibitem{Tudor95ap_processes}
C.~Tudor.
\newblock Almost periodic stochastic processes.
\newblock In {\em Qualitative problems for differential equations and control
  theory}, pages 289--300. World Sci. Publ., River Edge, NJ, 1995.

\bibitem{Tudor-Tudor99}
C.~A. Tudor and M.~Tudor.
\newblock Pseudo almost periodic solutions of some stochastic differential
  equations.
\newblock {\em Math. Rep. (Bucur.)}, 1(51)(2):305--314, 1999.


\bibitem{villani09oldnew}
C{\'e}dric Villani.
\newblock {\em Optimal transport}, volume 338 of {\em Grundlehren der
  Mathematischen Wissenschaften [Fundamental Principles of Mathematical
  Sciences]}.
\newblock Springer-Verlag, Berlin, 2009.
\newblock Old and new.

\bibitem{Vrkoc_95}
Ivo Vrko{\v{c}}.
\newblock Weak averaging of stochastic evolution equations.
\newblock {\em Math. Bohem.}, 120(1):91--111, 1995.

\bibitem{Wiener1926}
Norbert Wiener.
\newblock On the representation of functions by trigonometrical integrals.
\newblock {\em Mathematische Zeitschrift}, 24(1):575--616, 1926.

\bibitem{yan_existence_2015}
Zuomao Yan and Hongwu Zhang.
\newblock Existence of {Stepanov}-like square-mean pseudo almost periodic
  solutions to partial stochastic neutral differential equations.
\newblock {\em Ann. Funct. Anal. AFA}, 6(1):116--138, 2015.

\bibitem{Zaidmann71_Existence_Stepanov}
S.~{Zaidman}.
\newblock {An existence result for Stepanoff almost-periodic differential
  equations.}
\newblock {\em {Can. Math. Bull.}}, 14:551--554, 1971.

\bibitem{zhang94}
Chuan~Yi Zhang.
\newblock Pseudo-almost-periodic solutions of some differential equations.
\newblock {\em J. Math. Anal. Appl.}, 181(1):62--76, 1994.

\bibitem{zhang95}
Chuan~Yi Zhang.
\newblock Pseudo almost periodic solutions of some differential equations.
  {II}.
\newblock {\em J. Math. Anal. Appl.}, 192(2):543--561, 1995.

\bibitem{zhao_new_2011}
Zhi-Han Zhao, Yong-Kui Chang, and Gaston~M. N'Gu{\'e}r{\'e}kata.
\newblock A new composition theorem for {S}$^p$-weighted pseudo almost periodic
  functions and applications to semilinear differential equations.
\newblock {\em Opusc. Math.}, 31(3):457--474, 2011.

\end{thebibliography}
\def\cprime{$'$} \def\cprime{$'$}

\end{document}